\documentclass[final,hidelinks,onefignum,onetabnum]{siamart220329}

\newcommand{\edits}[1]{#1}
\newcommand{\fullver}[2]{#1}
\newcommand{\pedant}[2]{#1}
\makeatletter
\renewcommand*\env@matrix[1][c]{\hskip -\arraycolsep
  \let\@ifnextchar\new@ifnextchar
  \array{*\c@MaxMatrixCols #1}}
\makeatother

\usepackage{lipsum}
\usepackage{amsfonts}
\usepackage{graphicx}
\usepackage{epstopdf}
\usepackage{algorithmic}
\usepackage{subcaption}
\usepackage[nocompress]{cite}
\ifpdf

  \DeclareGraphicsExtensions{.eps,.pdf,.png,.jpg}
\else
  \DeclareGraphicsExtensions{.eps}
\fi

\newsiamremark{remark}{Remark}
\newsiamremark{hypothesis}{Hypothesis}
\crefname{hypothesis}{Hypothesis}{Hypotheses}
\newsiamthm{claim}{Claim}
\newsiamthm{assumption}{Assumption}

\headers{A Model-Free First-Order Method for LQR}{C. Ju, G. Kotsalis, and G. Lan}

\title{A Model-Free First-Order Method for Linear Quadratic Regulator with $\tilde{O}(1/\varepsilon)$ Sampling Complexity\thanks{Initially released on arXiv on Dec 1st, 2022.
\funding{This research was partially supported by NIFA grant 2020-67021-31526 and NSF grant 1909298. CJ is supported by the Department of Energy Computational Science Graduate Fellowship under Award Number DE-SC0022158.}}}

\author{Caleb Ju\thanks{H. Milton Stewart School of Industrial \& Systems Engineering, Georgia Institute of Technology, Atlanta, GA
  (\email{cju33@gatech.edu}, \email{george.lan@isye.gatech.edu}).}
\and Georgios Kotsalis\thanks{Amazon Core AI, Seattle, WA 
  (\email{kotsalis@gmail.com}), research conducted prior to joining Amazon.}
\and Guanghui Lan\footnotemark[2]}

\usepackage{amsopn}

\newcommand{\svec}{\mathrm{svec}}
\newcommand{\trace}{\mathrm{Tr}}
\newcommand{\T}{\mathrm{T}}
\newcommand{\Tr}{\mathrm{Tr}} 
\newcommand{\AR}[2]{\left[\begin{array}{#1}#2\end{array}\right]}
\newcommand{\hide}[2]{#1}
\newcommand{\TSig}{\widetilde{\Sigma}}

\ifpdf
\hypersetup{
  pdftitle={An Example Article},
  pdfauthor={D. Doe, P. T. Frank, and J. E. Smith}
}
\fi

\begin{document}

\maketitle

\begin{abstract}
We consider the classic stochastic linear quadratic regulator (LQR) problem under an infinite horizon average stage cost. 
By leveraging recent policy gradient methods from reinforcement learning, we obtain a first-order method that finds a stable feedback law whose objective function gap to the optima is at most $\varepsilon$ with high probability using $\tilde{O}(1/\varepsilon)$ samples, where $\tilde{O}$ hides polylogarithmic dependence on $\varepsilon$. 
Our method is the first \textit{online} (i.e., single trajectory) algorithm with this sampling complexity.
The improved dependence on $\varepsilon$ is achieved by showing the accuracy scales with the variance rather than the standard deviation of the gradient estimation error. 
Our developments that result in this improved sampling complexity fall in the category of actor-critic algorithms. 
The actor part involves a gradient descent-type method,
while in the critic part, we utilize a conditional stochastic primal-dual method and show that the algorithm has an accelerated rate of convergence when paired with a shrinking multi-epoch scheme. 
\end{abstract}

\begin{keywords}
linear quadratic regulator, actor-critic, primal-dual methods, non-convex, Markovian noise
\end{keywords}

\begin{MSCcodes}
93C05, 65K05
\end{MSCcodes}

\section{Introduction}
\sloppy The linear quadratic regulator (LQR) problem is a classic problem in control theory studied extensively since the work of Kalman~\cite{kalman1960general,kalman1960contributions}. Of interest is the infinite horizon linear-quadratic stochastic control problem.  Techniques for solving this problem go back to the 1960s. It is well known that when this system is  controllable, the optimal feedback law is a linear static state feedback of $u_t^* = -K^*x_t$ for some gain matrix $K^* \in \mathbb{R}^{m \times n}$, where $x_t \in \mathbb{R}^n$ is the current state. The matrix $K^*$ can be obtained by solving the algebraic Ricatti equation (ARE) \cite{anderson2007optimal,athans2013optimal}. 
For computational aspects in regards to the solution of AREs, as well as connections to linear matrix inequalities, the classic references are \cite{laub1994numeric}  and \cite{boyd1994linear}, respectively.
When the parameters of the system are unknown, one method is to obtain estimates  using system identification techniques~\cite{ljung1998system}  and the resulting uncertain system is then subject to robust control design~\cite{zhou1996robust}.

In recent years, methods from machine learning have been used for optimal control, in particular, for solving the LQR problem~\cite{tu2018least,Wainwright_pmlr_19,fazel2018global,sutton1988learning}. For example, one can estimate the system parameters using regularized least-squares. The estimated system is then viewed as the nominal system and the solution to the LQR problem is obtained by solving the corresponding ARE. This is known as the certainty equivalence principle~\cite{dean2020sample,mania2019certainty}, which is a \textit{model-based} method. In contrast, there are \textit{model-free} methods, which aim to solve the problem 
without explicitly estimating the system parameters. Many model-free methods are inspired by problems in reinforcement learning (RL)~\cite{sutton2018reinforcement}. RL tackles problems such as LQR using optimization methods that only require an oracle that can simulate the costs and evolution of the system. There are a plethora of model-free methods, and most can be categorized as either approximate dynamic programming~\cite{krauth2019finite,tu2018least} or gradient-based (first- or zeroth-order) methods~\cite{lan2022policy,li2022stochastic,schulman2017proximal}. 

We focus on gradient-based methods. It is known LQR is non-convex (the set of stable feedback laws is non-convex~\cite{fazel2018global}), so generally speaking at best we can only find locally optimal solutions. The ubiquitous gradient descent has been used in practice~\cite{wilson1986application} while its convergence to a local optimal solution was proved in~\cite{yan1994gradient}. It was not until the breakthrough result of~\cite{fazel2018global} that showed gradient descent converges to the \textit{global} optimal solution by proving the objective function of LQR satisfies the so-called Polyak-\L{}ojasiewicz-condition~\cite{lojasiewicz1963propriete,polyak1963gradient}, or P\L{}-condition.

In this paper, we develop a first-order method to efficiently solve LQR based on an actor-critic method~\cite{konda1999actor}. The outer loop, or the actor step, views the LQR problem as a nonlinear optimization problem with the P\L{}-condition, or a gradient domination condition.
To solve this, we present a method equivalent to natural gradient~\cite{kakade2001natural,fazel2018global}. In each actor step, one needs to solve the control Lyapunov equation, which is done by the critic. The critic solves this via a shrinking multi-epoch conditional stochastic primal-dual (CSPD) method. 

Let us summarize our contributions. For the actor, we provide a novel analysis of the natural gradient method for LQR, where we bound the gradient estimation error from the critic step in a new way so that the error squared -- rather than the error -- needs be at most the accuracy $\varepsilon$. This reduces the accuracy needed from the critic and is one of the key ideas to improve sampling complexity. For the critic, our newly proposed shrinking multi-epoch CSPD solves min-max problems with Markovian noise (i.e., data generated from an ergodic Markov chain).
By incorporating a primal predictive step to accelerate convergence, sampling procedure to mitigate temporal dependency, and multi-epoch scheme to reduce the overall error, our algorithm improves upon prior primal-dual methods and efficiently finds near optimal solutions. The resulting actor-critic method outputs a policy with function optimality gap at most $\varepsilon$ with probability $1-\delta$ using $O(\varepsilon^{-1}(\ln(1/\varepsilon) + \ln(1/\delta))^{7})$ samples (see~\cref{thm:a3_OPN} for the hidden constants). To the best of our knowledge, this result is new for the \textit{online} setting, where data is generated from a single trajectory of an ergodic Markov chain. Most importantly, our result uses minimal assumptions.

To justify our claim of a novel sampling complexity, let us review prior works. A similar actor-critic method yields $O(\varepsilon^{-5})$ sampling complexity~\cite{yang2019provably}, which was later improved to $O(\varepsilon^{-3/2})$~\cite{zeng2021two} and then $O(\varepsilon^{-1} (\ln\varepsilon^{-1})^2)$~\cite{zhou2022single}. However, the third listed method is not for the online setting since it generates multiple independent trajectories from some fixed state. Therefore, our algorithm is the first to obtain $\tilde{O}(\varepsilon^{-1})$ sampling complexity\footnote{$\tilde{O}$ hides polylogarithmic dependence on $\varepsilon$.} in the online setting, which may be a more realistic setting for systems that cannot be reset to a fixed state. Additionally, our result makes weaker assumptions than prior works. 
The latter two works~\cite{zeng2021two,zhou2022single} posit an almost sure bound on the mixing rate and norm of every policy generated by the algorithm, which may not be realistic since it is an assumption on the behavior of a stochastic algorithm. In contrast, we only assume the initial policy given to us is stable and $\varepsilon$ is not too large. We avoid these assumptions by obtaining high probability results and identifying a relationship between mixing rates, policy norm, and objective value.

Our discussion so far has focused on model-free methods. For model-based methods, a least-squares estimator can estimate the dynamics using $O(\varepsilon^{-2})$ samples~\cite{dean2020sample}. This was later improved to $O(\varepsilon^{-1})$ by noticing the error scales as the square of the estimation error~\cite{mania2019certainty}. However, this method also requires generating multiple independent trajectories and is therefore not in the online setting. Thus, our proposed method shows the model-free case can be competitive with the model-based case in terms of the dependence on $\varepsilon$, even without generating independent trajectories. Moreover, one may prefer model-free since model-based methods can suffer from \textit{model bias}, where insufficient samples and unaccounted-for uncertainties can lead to poorly fitted models and unstable policies~\cite{dean2020sample,deisenroth2011pilco}. 

The paper is organized as follows. In~\cref{sec:a4_LQR}, we formulate the LQR optimization problem, explore its properties, and cover related works. In~\cref{sec:a3_LQR}, 
we present a gradient method to solve the LQR problem.
This algorithm requires an accurate solution to the  \textit{control Lyapunov equation}, which is solved in~\cref{sec:a1_LQR} by using a shrinking multi-epoch stochastic primal-dual method. We finish with some preliminary numerical experiments in~\cref{sec:exper}.

\section{Preliminaries, problem formulation, and related works} \label{sec:a4_LQR}

\subsection{Notation}
We refer to the norm $\|\cdot \|$ as the standard $\ell_2$-norm for vectors and the induced $\ell_2$-norm for matrices. 
The norm $\| \cdot \|_F$ is the Frobenius norm for a given matrix. The notation $\rho(\cdot)$ refers to the spectral radius of a given matrix. For symmetric matrices $A$ and $B$, we use the standard Loewner order of $A - B \succeq 0$ to mean that $A - B$ is semidefinite and $A - B \succ 0$ to mean that it is positive definite. For any symmetric matrix $X$, let $\svec(X)$ be the vectorization of the upper triangular part of $X$ with the off-diagonals scaled by $\sqrt{2}$. Unless otherwise specified, we write $x = O(y)$ if there is an absolute constant $C$ (i.e., a scalar independent of $x$ and $y$) such that $x \leq C \cdot y$. 

\subsection{Linear systems and quadratic control} \label{sec:lin_sys_and_qc}

We consider  a discrete-time,  time-invariant linear dynamical system with disturbances,
\begin{equation}
\label{dynamics}
x_{t+1} = A x_t + B u_t + w_t,  \quad t \in \mathbb{Z}_+,
\end{equation}
where $x_t \in \mathbb{R}^n$ is the state, $u_t  \in \mathbb{R}^m$ is the control input, and $w_t \in \mathbb{R}^n$ is the stochastic disturbance, as well as $A \in \mathbb{R}^{n \times n}$ and $B \in \mathbb{R}^{n \times m}$. It is assumed that the random variables $\{ x_0, w_0, w_1, \hdots, \}$ are jointly independent and Gaussian 
($x_0 \sim \mathcal{N}(0, \Sigma_0)$ and $w_t \sim \mathcal{N}(0, \Psi)$)
where the covariance matrices $ \Sigma_0 \succ 0, \Psi \succ 0$ are known.
The infinite horizon time-average cost to be minimized over causal (i.e., nonanticipative) policies is 
\begin{equation}
\label{ergodic_cost}
\textstyle J = \lim_{T \rightarrow \infty} T^{-1}  \mathbb{E}\big[    \sum_{t=0}^{T-1} x^{\T}_t Q x_t +  
u^{\T}_t R u _t \big], 
\end{equation}
where $Q \in \mathbb{R}^{n \times n}$  and $R \in \mathbb{R}^{m \times m}$ are symmetric positive definite matrices. We assume the  matrices $(A,B)$ form a controllable pair. 
This is a classic problem in control theory whose solution can be reviewed in \cite{caines_18,varaiya_15}. 
The average cost is minimized by a linear state feedback policy
$u_t = - K^* x_t =  - (R + B^\T P B)^{-1}  B^\T P A x_t$
for times $t \in \mathbb{Z}_+$,
with $P \in \mathbb{R}^{n \times n}$ being the unique positive definite solution to the algebraic Riccati equation
$P = Q + A^\T P A - A^\T P B (R + B^\T P B)^{-1}  B^\T P A$.
The focus of this paper is on the model-free case, and we will restrict ourselves to randomized linear state feedback policies of the form 
\begin{equation} \label{eq:randomized_linear_state_feedback}
u_t = -Kx_t + v_t, ~~~ v_t \sim \mathcal{N}(0, \sigma^2 I_m),
\end{equation}
where $v_t$ is additive Gaussian noise with some noise level $\sigma \geq 0$ and $I_m$ is the identity matrix of dimension $m$. 
The additional feedback law disturbance, $v_t$, is often included to allow for exploration of the state and action space~\cite{dean2020sample,krauth2019finite,yang2019provably}. 
We also require that the basic random variables $\{ x_0, w_0, v_0, w_1, v_1 , \hdots, \} $ are jointly independent.
The resulting 
closed loop dynamics under \eqref{eq:randomized_linear_state_feedback} are
\begin{align} \label{eq:a17_LQR}
x_{t+1} = ( A - B K ) x_t + w_t + B v_t. 
\end{align}
We denote the set of stable policies by 
\begin{equation}
\label{stabilizing_gains}
\mathcal{S} = \{ K \in \mathbb{R}^{m \times n}~|~ \rho(A-B K) < 1 \}.
\end{equation}
We will search in an iterative fashion for a feedback gain $K \in \mathcal{S}$ that yields an approximation 
to the optimum within a prespecified degree of fidelity. 
First, we rewrite the cost $J$ as a function of $K$ in closed-form\footnote{The case $\sigma = 0$ (i.e., $u_t=-Kx_t$) is a classical result found in textbooks~\cite{caines_18}, while the case of $\sigma > 0$ can be derived by substituting in $u_t = -Kx_t + v_t$ into~\eqref{ergodic_cost} and using~\eqref{eq:a17_LQR} in place of~\eqref{dynamics} to deduce the steady state covariance matrix $\Sigma_K$.}, which yields the following optimization problem for LQR:
\begin{equation*} 
\min_K J(K) =  \begin{cases}  \Tr[ (Q + K^\T R K)  \Sigma_K + \sigma^2 R] = \Tr[ P_K(\Psi + \sigma^2 BB^{\T})+ \sigma^2 R],  & K \in \mathcal{S} \\
\infty, & \text{o/w}
\end{cases}.
\end{equation*} 
Denote $K^* \in \mathcal S$ as the optimal solution.
Here, the matrix $ \Sigma_K \in \mathbb{R}^{n \times n}$ is the solution to the Lyapunov equation
\begin{equation}
\label{eq:lyap_cov_mat}
\Sigma_K  = (\Psi + \sigma^2 BB^{\T}) + (A - B K) \Sigma_K (A - B K)^{T},
\end{equation}
and it corresponds to the steady state covariance matrix of the state vector. The matrix $P_K \in \mathbb{R}^{n \times n}$ is the unique positive definite solution to the Lyapunov equation
        $P_K = Q + K^{\T}  R K  + (A - B K)^{\T}  P_K (A - B K)$.
Now we will establish some properties pertaining to the objective value $J(K)$, mixing rate $\rho(A-BK)$, and policy norm. 
We defer the proof to~\cref{sec:pfs_for_basic_props}.
\begin{lemma} \label{lem:lqgp_a13}
    Let $\rho \equiv \rho(A-BK)$. If $K \in \mathcal S$, then
    \begin{align*}
        \|\Sigma_K\| &\leq \trace(\Sigma_K) \leq \sigma_{\min}(Q)^{-1} J(K) \\ 
        \|P_K\| &\leq \trace(P_K) \leq \sigma_{\min}(\Psi)^{-1} J(K) \\
        \|K\|_F^2 &\leq [{\sigma_{\min}(\Psi)\sigma_{\min}(R)}]^{-1}J(K)\\
        (1-\rho^2)^{-1} &\leq [\sigma_{\min}(\Psi) \sigma_{\min}(Q)]^{-1}J(K).
    \end{align*}
\end{lemma}

The following lemma contains an expression for the policy gradient $\nabla J(K) $.
\begin{lemma}
\label{policy_gradient_expression}
    For any $K \in \mathcal{S}$, the gradient of the cost function $J(K)$ is 
    $\nabla J(K) 
    = 
    2E_K \Sigma_K$
where the matrix $E_K$, the natural gradient, is given by
\begin{equation} \label{eq:E_K_def}
    E_K = (R + B^{\T} P_K B) K - B^{\T} P_K A. 
\end{equation}
\end{lemma}
The proof of this lemma can be traced to 
\cite{fatkhullin2021optimizing} with minor adjustments to account for the fact that we are working in discrete time.
The name \textit{natural gradient} is from the reinforcement learning literature~\cite{fazel2018global,kakade_NIPS_01}. 
Notice $E_K$ does not depend on the control input noise $\sigma$ used in~\eqref{eq:randomized_linear_state_feedback}.

If we can accurately estimate the gradient, we can characterize the change in objective value when we descend along the negative of this direction. The following result, also known as the performance difference lemma, quantifies this statement.
\begin{lemma} \label{lem:lqgp_a14}
    For any $K', K \in \mathcal S$, where $E_{K'}$ is the exact natural gradient for $K'$,
    \begin{align*}
        &J(K)-J(K') \\
        &= 2\trace(\Sigma_{K}(K-K')^TE_{K'}) + \trace(\Sigma_{K}(K-K')^T(R+B^TP_{K'}B)(K-K')).
    \end{align*}
\end{lemma}
A proof can be found in~\cite[Lemma 6]{fazel2018global}. 
The above result is akin to the smoothness of a function, where the change in $J(K')-J(K)$ can be upper bounded by a linear plus quadratic term. By fixing $K'=K^*$, the above result derives the so-called P\L{}-condition or gradient domination condition~\cite{lojasiewicz1963propriete,polyak1963gradient}. 
\begin{lemma} \label{lem:lqgp_a19}
    For any $K \in \mathcal S$,
    \begin{align*}
        \frac{\sigma_{\min}(\Psi)}{\|R + B^TP_{K}B\|} \|E_{K}\|^2_F 
        \leq
        J(K) - J(K^*) \leq \frac{\|\Sigma_{K^*}\|}{\sigma_{\min}(R)} \|E_{K}\|^2_F.
    \end{align*}
\end{lemma}
A proof can be readily deduced from 
\cite{fatkhullin2021optimizing}.
In view of~\cref{policy_gradient_expression}, an accurate estimate of $E_K$ requires an estimate of the solution $P_K$ from the Lyapunov equation
as well as the system parameters $R$, $B$, and $A$. Alternatively, one can estimate $E_K$ without such knowledge. 
First, we introduce the following differential action value function (Q-function),
    $Q_K(x,u) 
    := 
    \sum\limits_{t=0}^{\infty}( \mathbb{E}[c(x_t,u_t) \vert x_0=x, u_0 = u] - J(K) )$,
where $c(x,u) := x^TQx + u^TRu$ and the expectation is taken w.r.t.~\eqref{eq:randomized_linear_state_feedback} and~\eqref{eq:a17_LQR} for $t \geq 1$. 
The Q-function has two useful properties. First, it has a closed-form expression. 
Denote $z=[x,u]$ as a state-action pair, and analogously let $Q_K(z) = Q_K(x,u)$ and $c(z) = c(x,u)$. 
From~\cite[Proposition 3.1]{yang2019provably},
    $Q_K(z) = z^T\Theta(K)z - \sigma^2 \cdot \trace(R + P_KBB^T) - \trace(P_K\Sigma_K)$,
%
where the matrix $\Theta(K) \in \mathbb{R}^{(n+m) \times (n+m)}$ is defined as
\begin{align*} 
\Theta(K) = \AR{c}{A^{\T} \\ B^{\T}} P_K \AR{cc}{A & B} +  \AR{cc}{Q & 0 \\ 0 & R}.
\end{align*}
The above decomposition makes evident the fact that $\Theta(K) \succ 0$.
An important relation between the Q-function and natural gradient $E_K$ is seen by 
\begin{equation} \label{eq:natgrad}
   E_K = \Theta(K)_{22}     K - \Theta(K)_{21},
\end{equation}
where $\Theta(K)_{ij}$ is the $(i,j)$-th block of $\Theta_K$ if we view it as a $2 \times 2$ block matrix. 
Thus, by obtaining $\Theta(K)$, we can construct $E_K$.
We will show one way to do so by exploiting the second useful property: the average reward \textit{Bellman equations} (c.f.~\cite[Theorem 8.2.6]{puterman2014markov} and~\cite[Equation 2.4]{li2022stochastic}), which is a fixed-point equation:
\begin{align} \label{eq:bellman_eq}
    Q_K(z) = c(z)-J(K) + \mathbb{E}[Q_K(z') \vert z], \ \forall z \in \mathbb{R}^{n +m },
\end{align}
where $z'$ is the next state-action pair with gain matrix $K$ conditioned on the current state-action pair $z$.
Since this is an infinite-dimension linear system of equations, we reduce to a finite-dimension one. 
To do so, define
\begin{align} \label{eq:phi_def}
\phi(z) = \svec( zz^T ), \quad  
\theta(K) = \svec( \Theta(K)  ), \quad 
\vartheta(K) = \AR{c}{ J(K) \\ \theta(K)} \in \mathbb{R}^{{n + m \choose 2}+1}.
\end{align}
One can check $Q_K(z) = \phi(z)^T\theta(K) + (\text{constants independent of $z$})$. 
Denote the stationary measure on the state and control space w.r.t. to the current gain matrix $K$ as $\Pi_K$.
By multiplying~\eqref{eq:bellman_eq} by $\phi(z)$, taking expectation w.r.t.~$z \sim \Pi_K$, and re-arranging terms, we arrive at the linear system of equations,
\begin{equation} \label{eq:linsys_bellman}
    H \vartheta(K) = b,
\end{equation}
where the matrix and vector are defined as
\begin{equation} 
\label{linear_system_parameters}
	H = \AR{cc}{1 & 0 \\  \mathbb{E}_{\Pi_K}[ \phi(z)  ]   & \Xi_K },
	\ \ b = \AR{c}{ \mathbb{E}_{\Pi_K}[ c(z) ] \\ \mathbb{E}_{\Pi_K}[c(z) \phi(z)] },
\end{equation}
and
  $\Xi_K = \mathbb{E}_{\Pi_K}[ \phi(z)     ( \phi(z)  -   \phi(z') )^{T} ]$.
If $\sigma$ is sufficiently large, then the solution to~\eqref{eq:linsys_bellman} is unique (since $H$ is invertible from~\cref{lem:lqgp_a29}), and therefore it is also equal to the solution to~\eqref{eq:bellman_eq}.
So by solving for $\vartheta(K)$, we can extract $\Theta(K)$ and construct the natural gradient $E_K$ via~\eqref{eq:natgrad}.

\subsection{Related Works}
Two prominent paradigms of LQR being studied are the \textit{random initialization} and \textit{noisy dynamics} model~\cite{malik2019derivative}. The random initialization assumes only the initial state is random, and the remaining state transitions are deterministic. The noisy dynamics model, which we study in this paper, considers both a random initialization and perturbations in each state transition. 

Within the noisy dynamics setting, an infinite-horizon average cost was studied by~\cite{yang2019provably} using an actor-critic method. Around the same time,~\cite{krauth2019finite} used an approximate policy iteration from the RL literature. More recently,~\cite{zeng2021two,zhou2022single} utilize a stochastic gradient descent-type analysis to solve LQR. While they can get similar convergence rates, their results rely on a strong assumption that $\rho(A-BK_t)$ is almost surely less than one for every $t$, where $K_t$ is the gain matrix at iteration $t$. The aforementioned methods are model-free. As for model-based methods,~\cite{dean2020sample} used a least-squares estimator for the transition dynamics and then use tools from system-level synthesis to obtain their controller.~\cite{mania2019certainty} later improved the sampling complexity by noticing the function gap scales as the square of the parameter error. 

Similar problems include the infinite-horizon discounted objective value. The PhD thesis of~\cite{bradtke1994incremental} investigated the setting where there was no noise. In the noisy dynamic setting,~\cite{fiechter1997pac} developed 
algorithms that find near-optimal solutions with high probability guarantees. The authors of~\cite{tu2018least} studied the sampling complexity of a least-squares temporal difference in this setting. A question of minimizing a discount factor of one on the random initialization setting was first shown to be solvable to global optimality by the breakthrough paper of~\cite{fazel2018global}. Since their work, improvements in the sampling complexity have been proposed in the works of~\cite{malik2019derivative} and~\cite{mohammadi2021convergence}. It should be mentioned these last two works explored the use of zeroth-order methods. This setting, while model-free, requires a generative model, where one can start running the linear dynamics from any initial state. 

As mentioned in the introduction, our proposed shrinking multi-epoch conditional stochastic primal-dual algorithm can solve general min-max problems with Markovian noise.
Similar primal-dual methods, under the guise of gradient temporal differencing~\cite{wang2017finite,liu2020finite}, have been proposed. But our method leverages several new tools to improve upon the prior art. First, we incorporate a primal predictive step to improve convergence~\cite{lan2020first}. Second, we exploit a sharpness condition~\cite{polyak1979sharp} via the shrinking multi-epoch scheme to significantly reduce the deterministic error rates~\cite{ghadimi2013optimal}. Third, we introduce a new sampling procedure borrowed from conditional temporal differencing to improve dependence on the mixing rates~\cite{nagaraj2020least,kotsalis2022simple}.


\section{The actor: Policy optimization} \label{sec:a3_LQR}
Our goal is to solve the  optimization problem
\begin{align*}
    \min_{K \in \mathcal{S}} J(K),
\end{align*}
where recall $\mathcal{S}$ is the set of stable policies defined in~\cref{stabilizing_gains}. 
Since the optimal policy $K^*$ lies in the interior of the open set $\mathcal{S}$~\cite{fatkhullin2021optimizing}, a necessary condition of optimality is $\nabla J(K^*) = 2E_{K^*}\Sigma_{K^*} = 0$. 
\fullver{
A popular approach to minimize the gradient norm is by a gradient descent-type method.
The \textit{natural policy gradient} (NPG) from the reinforcement learning literature is a type of preconditioned gradient descent, where one moves in the negative direction of the natural gradient $E_{K_t}$ rather than the gradient $E_{K_t}\Sigma_{K_t}$~\cite{kakade_NIPS_01}; the advantages here are the need to only estimate $E_{K_t}$ and possibly faster convergence.}{
By observing that the covariance matrix from~\eqref{eq:lyap_cov_mat} satisfies $\Sigma_K \succ 0$, we can solve the variational inequality (VI):
\begin{align*}
    \text{find } K^* \text{ s.t. } \langle E_{K^*}, K-K^* \rangle \geq 0, \hspace{5pt} \forall K \in \mathbb{R}^{m \times n}.
\end{align*}
The solution to this VI problem is related to our original LQR problem thanks 
to the following generalized monotonicity condition~\cite{facchinei2003finite},
\begin{align*}
    2\langle \Sigma_{K^*} E_K, K - K^*\rangle 
    \geq 
    J(K) - J(K^*)
    \geq
    0,
\end{align*}
which is derived from~\cref{lem:lqgp_a14} by fixing $K=K^*$ and $K'=K$ and rearranging terms.
This VI shows the connection of our method to the VI formulations in reinforcement learning~\cite{lan2022policy,li2022homotopic,li2022stochastic}. 

By exploiting the monotonicity condition, we write an algorithm to solve the VI in~\cref{alg:b1_LQR}. This method is equivalent to the \textit{natural policy gradient method} seen in the reinforcement learning literature.

}
Each iteration of NPG first estimates the natural gradient $E_{K_t}$
by solving the linear system of equations in~\eqref{eq:linsys_bellman}. In the next section, we provide a method to complete this task. With $E_{K_t}$, we get the next policy $K_{t+1}$ (Line~\ref{line:pe_eqn}).
We repeat this process for a total of $T$ times and return the final policy $K_T$. 


\begin{algorithm}
\caption{Natural Policy Gradient Descent}
\label{alg:b1_LQR}
\begin{algorithmic}[1]
\STATE{Input: $K_0$, $\eta \in \mathbb{R}_{++}$}
\FOR{$t=0,\ldots,T-1$}
\STATE{Solve $H\begin{bmatrix}J(K_t) \\ \theta(K_t)\end{bmatrix} = b$ with~\cref{alg:a2_OPN} and form $E_{K_t} = \Theta(K_t)_{22}K_t - \Theta(K_t)_{21}$}
\STATE{$K_{t+1} = K_t - 2\eta E_{K_t}$} \label{line:pe_eqn} 
\ENDFOR
\RETURN $K_{T}$
\end{algorithmic}
\end{algorithm}

Let us now move on to proving the convergence of~\cref{alg:b1_LQR}. 
%
To ease notation in our analysis, we define the following positive constants:
    \begin{equation} \label{eq:b2_LQR}
    \begin{split}
        C_1 &:= 2[\|R\| + 2\sigma_{\min}^{-1}(\Psi)\|B\|^2 J(K_0)] \\
        C_2 &:= \sigma_{\min}(\Psi) \sigma_{\min}(Q)\sigma_{\min}(R) \\
        C_3 &:= J(K_0)/\sigma_{\min}(Q) \\
        C_4 &:= \|A\| + \|B\|\sqrt{2J(K_0)/[\sigma_{\min}(\Psi)\sigma_{\min}(R)]}.
    \end{split}
\end{equation}
We also define the natural gradient error,
\begin{align} \label{eq:a12_LQR}
    \delta_t := E_{K_t}^\star - {E}_{K_t},
\end{align}
where $E_{K_t}^\star$ is the true natural gradient and ${E}_{K_t}$ is the stochastic estimate. We will now show a contraction property for the LQR problem. 
\begin{proposition} \label{prop:pqgp_a13}
    Suppose $K_t \in \mathcal S$, $J(K_t) \leq 2J(K_0)$,~\cref{alg:b1_LQR} uses constant step size
        $\eta_t = 1/(2C_1)$, 
    and let $K_{t+1}$ be the computed (stochastic) policy. If $K_{t+1} \in \mathcal S$, then
    \begin{align} \label{eq:qgp_a51}
        J(K_{t+1}) - J(K^*) 
        \leq 
        \Big(1-\frac{C_2}{2J(K^*) C_1} \Big)[J(K_t) - J(K^*)] + \epsilon_t,
    \end{align}
    where $\frac{C_2}{2J(K^*)C_1} \in (0,\frac{1}{4}]$ and 
    \begin{align} \label{eq:qgp_a50}
        \epsilon_t := \frac{1}{C_1} \Big[ \|\delta_t\|_F^2\|\Sigma_{K_{t+1}}\| + \sigma_{\min}(\Psi) \trace([E_{K_t}^\star]^T\delta_t)\Big].
    \end{align}
\end{proposition}
\begin{proof}
    By the performance difference lemma (\cref{lem:lqgp_a14}) and definition of $K_{t+1}$,
    \begin{align*} 
        &J(K_{t+1}) - J(K_t)  \\
        &=
        -4\eta_t \trace(\Sigma_{K_{t+1}} {E}_{K_t}^T\delta_t) -4\eta_t \trace(\Sigma_{K_{t+1}} {E}_{K_t}^T{E}_{K_t}) \\ 
        &\hspace{11pt} + 4\eta_t^2  \trace(\Sigma_{K_{t+1}} E_{K_t}^T (R + B^TP_{K_t}B) E_{K_t}) \\
        &\stackrel{(*)}{\leq}
        2\eta_t \|E_{K_t}S^T\|_F^2 + 2\eta_t \|\delta_tS^T\|_F^2 - (4\eta_t - 4\eta_t^2 \|R + B^TP_{K_t}B\|) \trace(\Sigma_{K_{t+1}} E_{K_t}^TE_{K_t}) \\
        &=
        2\eta_t \|\delta_t S^T\|_F^2 - (2\eta_t - 4\eta_t^2 \|R + B^TP_{K_t}B\|) \trace(\Sigma_{K_{t+1}} E_{K_t}^TE_{K_t}) \\
        &\leq
        \frac{1}{C_1} \|\delta_t\|_F^2\|\Sigma_{K_{t+1}}\| - \frac{1}{2C_1} \trace(\Sigma_{K_{t+1}} E_{K_t}^TE_{K_t}) \\
        &\leq
        \frac{1}{C_1} \|\delta_t\|_F^2\|\Sigma_{K_{t+1}}\| - \frac{\sigma_{\min}(\Psi)}{2C_1} \trace(E_{K_t}^TE_{K_t}),
    \end{align*}
    where $\Sigma_{K_{t+1}} = S^TS$ (since $\Sigma_{K_{t+1}}$ is symmetric positive semidefinite, or sym-psd). 
    The first inequality used Cauchy-Schwarz and Young's inequality to show $\vert \trace(M^TN) \vert \leq \|M\|_F\|N\|_F \leq \frac{\|M|\|_F^2}{2} + \frac{\|N\|_F^2}{2}$ for matrices $M$ and $N$, and $\vert \trace(MN) \vert \leq \|M\|\trace(N)$ when $N$ is sym-psd. 
    The second inequality is by $\|\delta_t S^T\|^2_F = \trace(\Sigma_{K_{t+1}} \delta_t^T \delta_t)$ and
    \begin{align} \label{eq:a21_LQR}
         \|R + B^TP_{K_t}B\|
         \leq
         \|R\| + \|B\|^2\|P_{K_t}\|
         \leq
         \|R\| + \sigma_{\min}^{-1}(\Psi)J(K_t)\|B\|^2
         \leq
         C_1/2,
    \end{align}
    which is due to $J(K_t) \leq 2J(K_0)$ and~\cref{lem:lqgp_a13}. In the third inequality, we used $\trace(MN) \geq \sigma_{\min}(M)\trace(N)$ when $N$ is sym-psd and $\Sigma_{K_{t+1}} \succeq \Psi \succ 0$ (c.f.~\eqref{eq:lyap_cov_mat}). 
    Noting $\trace(\delta_t^T\delta_t) \geq 0$, we have $\trace(E_{K_t}^TE_{K_t}) \geq  \trace([E_{K_t}^\star]^TE_{K_t}^\star) - 2\trace([E_{K_t}^\star]^T\delta_t)$. Altogether, 
    \begin{equation*} \label{eq:qgp_a46}
    \begin{split}
        &J(K_{t+1}) - J(K_t)   \\
        &\leq
        -\frac{\sigma_{\min}(\Psi)}{2C_1}\trace([E_{K_t}^\star]^TE_{K_t}^\star) 
        + \frac{1}{C_1} \Big[ \|\delta_t\|_F^2\|\Sigma_{K_{t+1}}\| + \sigma_{\min}(\Psi) \trace([E_{K_t}^\star]^T\delta_t)\Big] \\
        &\leq -\frac{C_2}{2 J(K^*) C_1}[J(K_t)-J(K^*)] + \epsilon_t.
    \end{split}
    \end{equation*}
    where in the last line we used the P\L{}-condition (\cref{lem:lqgp_a19}) followed by the bound on $\|\Sigma_{K_*}\|$ (\cref{lem:lqgp_a13}). Adding $J(K_t)-J(K^*)$ to both sides finishes the proof for~\eqref{eq:qgp_a51}. 
    
    The bound of $\frac{C_2}{2J(K^*)C_1} \in (0,\frac{1}{4}]$ follows by noticing all the constants are positive and bounded, and using $C_1$ and $C_2$ from~\eqref{eq:b2_LQR},~\cref{lem:lqgp_a13}, and $K^* \in \mathcal S$ to lower bound $2J(K^*)C_1 \geq 4J(K^*)\|R\| \geq 4\sigma_{\min}(R)\sigma_{\min}(\Psi)\sigma_{\min}(Q)/(1-\rho(A-BK^*)) \geq 4C_2$.
\end{proof}
The key idea of this result, highlighted by the inequality marked with $(*)$ in the proof, is to bound the error term $\mathrm{Tr}(\Sigma_{K_{t+1}} E_{K_t}^T\delta_t)$ by the squared norm $O(\|\delta_t\|_F^2)$ rather than $O(\|\delta_t\|_F)$ (c.f., the error scales with $O(\|\delta_t\|_F)$ in~\cite[Eq. (5.64)]{yang2019provably}). 
We then only need to bound $\|\delta_t\|_F^2 \leq \epsilon$, which is a weaker requirement than $\|\delta_t\|_F \leq \epsilon$ when $\epsilon$ is small.

The proposition above by itself is incomplete, though. First, it assumes $K_{t+1} \in \mathcal S$ and second, the error $\epsilon_t$ depends on $\|\Sigma_{K_{t+1}}\|$. To resolve these issues, we introduce two results, adapted from~\cite{fazel2018global} as Lemma 15 and 16.

\begin{lemma} \label{lem:a2_LQR}
    If $K_t \in \mathcal S$, $E_{K_t}^\star$ is the exact natural gradient for $K_t$, and the step size satisfies $\eta_t \leq \|R + B^TP_{K_t}B\|^{-1}$, then $K_{t+1}^\star =K_t - 2\eta_tE_{K_t}^\star \in \mathcal S$.
\end{lemma}
\begin{lemma} 
\label{lem:lqgp_a32}
    Suppose $K \in \mathcal S$ and
       $\|K' - K\| \leq \frac{\sigma_{\min}(Q)\sigma_{\min}(\Sigma_K)}{4J(K)\|B\|(\|A-BK\|+1)}$
    for some matrix $K'$.
    Then $K' \in \mathcal S$ and 
    $
        \|\Sigma_{K'} - \Sigma_K\| 
        \leq 
        4 \Big( \frac{J(K)}{\sigma_{\min}(Q)} \Big)^2 \frac{\|B\|(\|A-BK\| +1)}{\sigma_{\min}(\Sigma_K)} \|K' - K\|$.
\end{lemma}
We will now apply the above results in the following proposition. 
Similar to $K_{t+1}$ from Line~\ref{line:pe_eqn}, denote the policy update with exact natural gradient as $K_{t+1}^\star = K_t - 2\eta_t E_{K_t}^\star$.

\begin{proposition} \label{lem:lqgp_a24}
    Suppose the assumptions from~\cref{prop:pqgp_a13} hold. Also suppose $\|K_{t+1} - K_{t+1}^\star\|$ satisfies the bound in~\cref{lem:lqgp_a32}. 
    Then 
    \begin{align*}
        J(K_{t+1}) - J(K^*) 
        \leq 
        \Big(1-\frac{C_2}{4J(K^*) C_1} \Big)[J(K_t) - J(K^*)] + \epsilon_t',
    \end{align*}
    where 
    \begin{align} \label{eq:a20_LQR}
        \epsilon_t' 
        &=
        \frac{1}{C_1} \bigg[ \frac{\sigma_{\min}(\Psi) C_1 J(K^*)}{2C_2} \|\delta_t\|_F^2 + 2C_3\|\delta_t\|_F^2 + \frac{16C_3^2\|B\|(1+C_4)}{\sigma_{\min}(\Psi) C_1} \|\delta_t\|^3_F \bigg].
    \end{align}
\end{proposition}
\begin{proof}
    Using~\eqref{eq:a21_LQR}, we get $\eta_t = (2C_1)^{-1} \leq (4\|R + B^TP_{K_t}B\|)^{-1}$.
    Then~\cref{lem:a2_LQR} tells us $K_{t+1}^\star \in \mathcal S$. 
    Now, using~\cref{lem:lqgp_a13} to bound $\|K_{t+1}^\star\|_F$,
    \begin{equation} \label{eq:qgp_a58}
    \begin{split}
        \frac{\|B\|(\|A-BK_{t+1}^\star\| +1)}{\sigma_{\min}(\Sigma_{K_{t+1}^\star})}
        &\leq 
        \frac{\|B\|(\|A\| + \|B\|\|K_{t+1}^\star\|_F +1)}{\sigma_{\min}(\Psi)}  \\
        &\leq 
        \frac{\|B\|(\|A\| + \|B\|\sqrt{J(K_{t+1}^\star)/[\sigma_{\min}(\Psi)\sigma_{\min}(R)]}  +1)}{\sigma_{\min}(\Psi)} \\
        &\leq 
        \frac{\|B\|(C_4 +1)}{\sigma_{\min}(\Psi)}.
    \end{split}
    \end{equation}
    where we used $\Sigma_K \succeq \Psi \succ 0$ from~\eqref{eq:lyap_cov_mat} in the first line and $J(K_{t+1}^\star) \leq J(K_t) \leq 2J(K_0)$ (from~\cref{prop:pqgp_a13}) and $C_4$ from~\eqref{eq:b2_LQR} in the last line. 
    In addition,~\cref{lem:lqgp_a13} also derives $\|\Sigma_{K_{t+1}^\star}\| \leq \frac{J(K_{t+1}^\star)}{\sigma_{\min}(Q)} \leq \frac{2J(K_0)}{\sigma_{\min}(Q)} =2C_3$. 
    Then upon applying~\cref{lem:lqgp_a32},
    \begin{align*}
        \|\Sigma_{K_{t+1}}\| 
        &\leq  
        \|\Sigma_{K_{t+1}^\star}\| + \|\Sigma_{K_{t+1}} - \Sigma_{K_{t+1}^\star}\| \\
        &\leq
        2C_3 + 4(2C_3)^2 \cdot \frac{\|B\|(1 + C_4)}{\sigma_{\min}(\Psi)} \|2\eta(E_{K_t} - E_{K_t}^\star)\| \\
        &\leq
        2C_3 + \frac{16C_3^2\|B\|(1 + C_4)}{\sigma_{\min}(\Psi) C_1} \|\delta_t\|,
    \end{align*}
    with $\delta_t = E_{K_t}^\star - E_{K_t}$ from~\eqref{eq:a12_LQR}.
    We have bounded the first term in $\epsilon_t$ from~\eqref{eq:qgp_a50}.
    We now bound the second term: $\sigma_{\min}(\Psi) \trace([E_{K_t}^\star]\delta_t)/C_1$. By Young's inequality,
    \begin{align*}
        \trace([E_{K_t}^\star]^T \delta_t) 
        &\leq 
        \frac{C_2}{4J(K^*)\|R + B^TP_{K_t}B\|}\cdot\|E_{K_t}^\star\|_F^2 + \frac{J(K^*)\|R + B^TP_{K_t}B\|}{C_2} \cdot \|\delta_t\|_F^2 \\
        &\leq
        \frac{C_2}{4\sigma_{\min}(\Psi) J(K^*)}\cdot [J(K_t)-J(K^*)] + \frac{J(K^*)C_1}{2C_2} \cdot \|\delta_t\|_F^2,
    \end{align*}
    \sloppy where in the last line we used the P\L{}-condition (\cref{lem:lqgp_a19}) and~\eqref{eq:a21_LQR} to bound the two summands. We finish by applying both bounds above back into~\eqref{eq:qgp_a51}. 
\end{proof}

We can now prove convergence results for~\cref{alg:b1_LQR}. 
To simplify notation, define the condition number $\kappa =  8J(K^*)C_1/C_2 \geq 16$ (see~\cref{prop:pqgp_a13}).
\begin{theorem} \label{thm:b1_LQR}
    \sloppy Let $l := \lceil \kappa \rceil$ and $\varepsilon \in (0, J(K_0)]$. If $K_0 \in \mathcal S$, step size is $\eta_t = 1/(2C_1)$, and gradient error from~\eqref{eq:a12_LQR} satisfies $\|\delta_t\|_F^2 \leq \min\{C_5,C_6,C_7\}$, where
    \begin{equation} \label{eq:13}
    \begin{split}
        C_5
        &:= 
        \Big( \frac{\|B\| \sigma_{\min}(Q)}{2(C_4+1)} \Big)^2,  \hspace{5pt} 
        C_6
        :=
        \Big( \frac{\sigma_{\min}(\Psi) \cdot C_2}{1920 J(K_0)\|B\|C_3^2(C_4+1)} \cdot \varepsilon \Big)^{2/3}, \\
        C_7 
        &:=\frac{C_2}{60J(K_0)}
        \min \Big( \frac{C_2}{ \sigma_{\min}(\Psi) J(K^*)  C_1} , \frac{1}{4C_3} \Big) \cdot \varepsilon, \\
    \end{split}
    \end{equation}
    then the following holds for every $t = 0, \ldots,T$:
    \begin{enumerate}
        \item Stability: $K_t \in \mathcal S$.
        \item Monotonicity or convergence: either $J(K_{t}) \leq J(K_{t-1})$ or $J(K_t) - J(K^*) \leq \varepsilon$.
        \item \sloppy Linear rate: $J(K_t) - J(K^*) \leq 2^{- \lfloor t/l \rfloor}J(K_0)$ as long as $t \leq l \cdot \log_2(J(K_0)/\varepsilon)$.
    \end{enumerate} 
\end{theorem}
\begin{proof}
    We use mathematical induction, where the base case of $t=0$ holds by our assumptions. Now consider some $t +1 \leq T$.
    We first show $K_{t+1} \in \mathcal S$. 
    Recall $K_{t+1}$ from Line~\ref{line:pe_eqn} and the error-free update, $K_{t+1}^\star = K_t - 2\eta E_{K_t}^\star$ where $\eta = 1/(2C_1)$. 
    Then
    \begin{align*}
        \|K_{t+1}-K_{t+1}^\star\|
        =
        \frac{\|\delta_t\|_F}{C_1}
        &\leq
        \frac{\|B\| \sigma_{\min}(Q)}{[4\sigma_{\min}^{-1}(\Psi)\|B\|^2J(K_0)] \cdot 2(C_4+1)} \\
        &\leq
        \frac{\sigma_{\min}(Q)}{4J(K_t) }\cdot \frac{\sigma_{\min}(\Psi)}{\|B\| (C_4+1)} \\
        &\leq
        \frac{\sigma_{\min}(Q)}{4J(K_t)} \cdot \frac{\sigma_{\min}(\Sigma_{K_t})}{\|B\|(\|A-BK_t\|+1)},
    \end{align*}
    \sloppy where the first inequality is from recalling $C_1$ from~\eqref{eq:b2_LQR} and using $\|\delta_t\|^2_F \leq C_5$, the second inequality is from $J(K_t) \leq \max\{J(K_0), J(K^*) + \varepsilon\} \leq 2J(K_0)$ (since $\varepsilon \leq J(K_0)$), and the third inequality can be shown similarly to~\eqref{eq:qgp_a58}. Therefore, the hypothesis for~\cref{lem:lqgp_a32} is satisfied, and so $K_{t+1} \in \mathcal S$.
    
    Next, we prove monotonicity or convergence. Suppose $J(K_{t}) - J(K^*) > \varepsilon$. By the inductive hypothesis, $K_t \in \mathcal S$ and we already showed $J(K_t) \leq 2J(K_0)$ and $K_{t+1} \in \mathcal S$. 
    Invoking~\cref{lem:lqgp_a24} and adding $J(K_{*})-J(K_t)$,
    \begin{align} \label{eq:b3_LQR}
        J(K_{t+1}) - J(K_t)
        &\leq 
        -\frac{C_2}{4J(K^*)C_1} [J(K_t) - J(K^*)] + \epsilon_t'
    \end{align}
    where $\epsilon_t'$ is defined in~\eqref{eq:a20_LQR}. By our bounds $\|\delta_t\|_F^2 \leq \min\{C_6, C_7\}$, we ensure 
        $\epsilon_t' 
        \leq 
        \frac{C_2 \varepsilon}{40J(K_0)C_1}
        \leq
        \frac{C_2 \varepsilon}{40J(K^*)C_1}$.
    Plugging back into~\eqref{eq:b3_LQR} and noticing $J(K_t) - J(K^*) > \varepsilon$, we get $J(K_{t+1}) \leq J(K_t)$. 
    
    \sloppy Finally, let us show linear rate of convergence. To simplify our notation, let us define the constant $Z = \frac{C_2}{5C_1}$. We already showed a bound on $\epsilon_t'$ in the previous paragraph, which is equivalent to $\epsilon_t' \leq \frac{Z \varepsilon}{8J(K_0)}$. Recalling that we consider iterations $t \leq l \cdot \log_2(J(K_0)/\varepsilon)$, this implies $\epsilon_{t}' \leq 2^{-(\lfloor t/l \rfloor -2)} (Z/2)$. 
    Thus, using our bound on $\epsilon_t'$ and choice of $l$, the recursion~\eqref{eq:b3_LQR} can be simplified into (c.f.,~\cite[Lemma 11]{lan2022policy})
    \begin{align*}
        J(K_{t+1}) - J(K^*) 
        &\leq 
        2^{-\lfloor (t+1)/l \rfloor} \Big( J(K_0) - J(K^*) + \frac{5l \cdot Z}{16} \Big) 
        \leq
        2^{-\lfloor (t+1)/l \rfloor} J(K_0),
    \end{align*}
    where the last inequality is because $\kappa \geq 1$, so $l/2 = \lceil \kappa \rceil /2 \leq \kappa$. This completes the proof of convergence and finishes our proof by induction. 
\end{proof}
A few remarks are in order. The assumption $\varepsilon \leq J(K_0)$ is mild, as otherwise the function gap is at most $\varepsilon$ with the initial policy.
From the theorem, we need $O(J(K_0)^2 \ln(J(K_0)/\varepsilon))$ 
iterations  to ensure $J(K_T) - J(K^*) \leq \varepsilon$, where big-O hides dependence on all constants except on $J(K_0)$ and $\varepsilon$. 
And an important consequence of the theorem is a uniform bound on the spectral radius of $A-BK_t$.  
\begin{corollary} \label{cor:qgp_a1}
    If~\cref{thm:b1_LQR}'s assumptions hold, then there exists a uniform bound on the spectral radius,
        $\max_{0 \leq t \leq T} \rho(A-BK_t) \leq \sqrt{1 - \frac{\sigma_{\min}(\Psi)\sigma_{\min}(Q)}{2J(K_0)}} < 1$.
\end{corollary}
The corollary can be proven with~\cref{lem:lqgp_a13} and $J(K_t) \leq 2J(K_0)$ from~\cref{thm:b1_LQR}.
Next, we will define the critic method to make the gradient error sufficiently small.

\section{The critic: Policy evaluation via stochastic primal-dual} \label{sec:a1_LQR}
Let $K \in \mathcal S$. 
Recall from~\cref{eq:linsys_bellman} the natural gradient $E_{K}$ can be derived by solving the linear system of equations induced by Bellman's fixed-point equation,
\begin{equation} \label{eq:qgp_a54}
    \mathbb{E}_{\Pi_K}[\tilde{H}]\vartheta = \mathbb{E}_{\Pi_K}[\tilde{b}],
\end{equation}
where $\tilde{H}$ and $\tilde{b}$ are the stochastic estimates of $H = \mathbb{E}_{\Pi_K}[\tilde{H}]$ and $b = \mathbb{E}_{\Pi_K}[\tilde{b}]$ as defined in~\cref{linear_system_parameters}. We re-formulate the above problem as a residual minimization problem $\min_{\vartheta \in X} \|\mathbb{E}_{\Pi_K}[\tilde{H}\vartheta-\tilde{b}]\|$, which is equivalent to the min-max problem,
\begin{align} \label{eq:a10_OPN}
    \min_{\vartheta \in X} \big \{ f(\vartheta) := \max_{y \in Y} \{\langle y, \mathbb{E}_{\Pi_K}[\tilde{H}]\vartheta - \mathbb{E}_{\Pi_K}[\tilde{b}]\rangle \} \big\} ,
\end{align}
where $X = \mathbb{R}^{{n \cdot m \choose 2}+1}$ and $Y = \{y \in \mathbb{R}^N: \|y\| \leq 1\}$ are the primal and dual spaces, respectively. 
The main advantage of the min-max formulation is one can efficiently get nearly unbiased estimates of the gradient, while one cannot as easily when using the $\ell_2$ norm from the residual minimization problem.

\edits{For the remainder of this section, we will establish some basic structural properties for $\tilde{H}$ and $\tilde{b}$, followed by an efficient primal-dual algorithm which leverages these properties to solve~\eqref{eq:a10_OPN}.
We finish by combining the primal-dual algorithm for policy evaluation with the previous natural policy gradient method to solve LQR.
}

%

\subsection{Structural properties of Bellman's fixed-point equation}

Recall $x_t \in \mathbb{R}^n$ is the state at time $t$ of the linear dynamical system~\eqref{dynamics} while $u_t = -Kx_t + v_t \in \mathbb{R}^m$ is the feedback, where $v_t$ is the policy disturbance introduced in~\eqref{eq:randomized_linear_state_feedback}. We let the stochastic process be $\{\xi_t  \in \mathbb{R}^{2(n+m)}: \xi_t = [x_t,u_t,x_{t+1},u_{t+1}]\}_t$, i.e., the current and next state-action pairs, and it is generated by some probability measure $(\Omega, \mathcal{F}, \mathbb{P})$ on the state space $\mathcal{S} \equiv \mathbb{R}^{2(n+m)}$. 
Define $\mathcal{F}_t = \sigma(\xi_1,\xi_2,\ldots,\xi_t)$ as the $\sigma$-algebra generated by the first $t$ random variables. 
Since the stochastic matrix/vector $(\tilde{H}, \tilde{b})$ from~~\cref{linear_system_parameters} are functions of two consecutive state-action pairs $(x_t,u_t)$ and $(x_{t+1},u_{t+1})$, 
then we write the stochastic estimates as
    $\tilde{H}_t = \tilde{H}(\xi_t) \text{ and } \tilde{b}_t = \tilde{b}(\xi_t)$
to indicate their parameterization by the random variable at time $t$. 

Recall $\Psi$ is the covariance matrix of the noise $w_t$ 
and define the covariance matrices
\begin{align} \label{eq:a15_LQR}
    \TSig_K^{(t)} =
    \begin{bmatrix}
        \Sigma_K^{(t)} & -\Sigma_K^{(t)} K^T \\
        -K \Sigma_K^{(t)} & K\Sigma_K^{(t)} K^T  + \sigma^2 I_m
    \end{bmatrix},
    \Sigma_K^{(t)} = \sum\limits_{p=0}^{t-1} (A-BK)^p \Psi [(A-BK)^T]^p,
\end{align}
where $I_m$ is the identity matrix of size $m$.
By examining the policy~\eqref{eq:randomized_linear_state_feedback} and dynamics~\eqref{eq:a17_LQR}, we can deduce the $t$-th state-action pair $[x_t,u_t]$ is distributed according to a multivariate Gaussian with mean $[x_0, -Kx_0]$ and covariance matrix $\TSig_K^{(t)}$.

Because $[x_t,u_t]$ has unbounded support, then so do $\tilde{H}(\xi_t)$ and $\tilde{b}(\xi_t)$. 
Thus, the stochastic estimates $(\tilde{H}(\xi_t), \tilde{b}(\xi_t))$ can be arbitrarily large, in which case the stochastic estimation error is also arbitrarily large.
To avoid such cases, we leverage the fact $[x_t,u_t]$ is a multivariate Gaussian, which exhibits the following ``light-tail'' property. 
This can be proven by the Hanson-Wright inequality (see~\cite{rudelson2013hanson}).
\begin{lemma} \label{prop:pqgp_a9}
    Let $\ell \sim \mathcal{N}(\mu, \Sigma)$. For any $\delta \in (0,1)$, we have
    \begin{align*}
        \mathbb{P} \Big \{ \frac{1}{2}\|\ell\|^2 > \mathrm{Tr}(\Sigma) + \frac{\sqrt{c_2^4\|\Sigma\|_F^2 \log({2}/{\delta})}}{\sqrt{c_1}}  + \frac{c_2^2\|\Sigma\|}{c_1}\log(\frac{2}{\delta}) + \|\mu\|^2 \Big \} 
        \leq
        \delta,
    \end{align*}
    where $c_1$ and $c_2$ are some absolute positive constants.
\end{lemma}
\hide{}{
\begin{proof}
    Let $z$ be a vector consisting of independent standard Gaussian random variables. 
    Then $z \sim \mathcal{N}(0,I)$. 
    By the Hanson-Wright inequality~\cite[Theorem 1]{rudelson2013hanson}, we have for any $\alpha \geq 0$,
    \begin{align*}
        \mathrm{Pr}\big\{ \vert z^T\Sigma z - \underbrace{\mathbb{E} \mathrm{Tr}(z^T\Sigma z)}_{\mathrm{Tr}(\Sigma)} \vert > \alpha \big\}
        \leq
        2\mathrm{exp}\big\{ -c_1 \cdot \min \big( \frac{\alpha^2}{c_2^4 \|\Sigma\|_F^2}, \frac{\alpha}{c_2^2\|\Sigma\|} \big) \big\},
    \end{align*}
    where $c_1$ and $c_2$ are some absolute constants (see~\cite[Theorem 1]{rudelson2013hanson} for more details).
    Setting $\alpha = \sqrt{\frac{c_2^4\|\Sigma\|_F^2}{c_1} \log(\frac{2}{\delta})} + \frac{c_2^2\|\Sigma\|}{c_1}\log(\frac{2}{\delta})$,
    \begin{align*}
        \mathrm{Pr}\Big\{ z^T\Sigma z > \mathrm{Tr}(\Sigma) + \sqrt{\frac{c_2^4\|\Sigma\|_F^2}{c_1} \log(\frac{2}{\delta})} + \frac{c_2^2\|\Sigma\|}{c_1}\log(\frac{2}{\delta}) \Big\} 
        \leq
        \delta.
    \end{align*}
    Now, define the random multivariate Gaussian, $\ell_0 := \Sigma^{1/2}z$.
    Observing $z^T\Sigma z = \ell_0^T\ell_0 = \|\ell_0\|^2$ and using the inequality $\frac{1}{2}\|\ell_0+\mu\|^2 \leq \|\ell_0\|^2 + \|\mu\|^2$, we conclude
    \begin{align*}
        &\mathrm{Pr}\Big\{ \frac{1}{2}\|\ell_0+\mu\|^2 > \mathrm{Tr}(\Sigma) + \sqrt{\frac{c_2^4\|\Sigma\|_F^2}{c_1} \log(\frac{2}{\delta})} + \frac{c_2^2\|\Sigma\|}{c_1}\log(\frac{2}{\delta}) + \|\mu\|^2 \Big\}  \\
        &\leq
        \mathrm{Pr}\Big\{ \|\ell_0\|^2 + \|\mu\|^2 > \mathrm{Tr}(\Sigma) + \sqrt{\frac{c_2^4\|\Sigma\|_F^2}{c_1} \log(\frac{2}{\delta})} + \frac{c_2^2\|\Sigma\|}{c_1}\log(\frac{2}{\delta}) + \|\mu\|^2 \Big\} 
        \leq
        \delta.
    \end{align*}
    The proof is finished upon noticing $\ell := \ell_0 + \mu$ is distributed wrt $\ell \sim \mathcal{N}(\mu, \Sigma)$.
\end{proof}
}

\hide{}{Note that when the current policy is $K$, assumed to be stable, then the dynamics of the system evolve according to (c.f.,~\eqref{eq:a17_LQR})
\begin{align*}
    X_{t+1} = (A - BK)X_t + w_t.
\end{align*}
Consequently, if we condition on some initial state $X_0$ (which can be the initial state of the system or the state we left off at during the last epoch, rollout, etc.), then $X_t \sim \mathcal{N}\big((A-BK)^tX_0, \Sigma_K^{(t)}\big)$, where the covariance is
\begin{align*}
    \Sigma_K^{(t)} := \sum\limits_{p=0}^{t-1} (A-BK)^p \Psi [(A-BK)^T]^p.
\end{align*}
Indeed, the asymptotic covariance, i.e., as $t \to \infty$ is equivalent to the solution of the Lyapunov equation of~\eqref{eq:lyap_cov_mat}: 
\begin{align} \label{eq:a16_LQR}
    \lim_{t \to \infty} \Sigma_K^{(t)} = \Sigma_K
\end{align}
In view of these definitions, we can now define the set of desirable events. 

Since $\|X\|_F^2 = \trace(X^TX) \leq \|X\|\|X\|_F$ for any square matrix $X$, then
\begin{align*}
    4\|\Sigma_K^{(t)}\| \sqrt{\ln\big(\frac{1}{\delta}}\big) + \frac{\|\Sigma_K^{(t)}\|_F^2}{4\|\Sigma_K^{(t)}\|} + \trace(\Sigma_K^{(t)}) + 2\|X_0\|^2
    \leq 
    \big(4\|\Sigma_K^{(t)}\| + \trace(\Sigma_K^{(t)}) + 2\|X_0\|\big) \ln \sqrt{\frac{e}{\delta}}.
\end{align*}
}

Consider the event $[x_t,u_t]$ is bounded for some $\delta \in (0,1/e]$:
\begin{align} \label{eq:a1_LQR}
    \mathcal{E}_t(\delta) = 
    \Big \{ \|[x_t, u_t]\|^2
    \leq  
    4\big(\frac{c_2^2}{\sqrt{c_1}} \|\TSig^{(t)}_K\| + \frac{c_2^2}{c_1}\trace(\TSig^{(t)}_K) + \|[x_0,  -Kx_0 ]\|^2 \big) \ln \frac{1}{\delta} \Big\}.
\end{align}
By~\cref{prop:pqgp_a9}, we know $\mathrm{Pr}(\mathcal{E}_t(\delta)) \geq 1-\delta$. 
\edits{Let $\tau \geq 1$ be some parameter we call the \textit{mixing time}, which will be further elaborated in~\cref{lem:lqgp_a22} and in \cref{sec:general_pd}.
For the meantime, one can view this as a ``buffer time'' between state-action pairs to reduce their correlation}.
Let us define the intersection of the following events:
\begin{align} \label{eq:qgp_a57}
    \textstyle \mathcal{E}(\delta) = \bigcap_{1 \leq t \leq \bar N} (\mathcal{E}_{t \cdot (\tau+1)-1}(\delta) \cap \mathcal{E}_{t \cdot (\tau+1)}(\delta)).
\end{align}
where $\bar N \geq 1$. While the dependence on $\bar N$ is omitted, it will be clear from context that $\bar N$ is the total number of observed samples $\xi_t$. 
While the analysis of policy evaluation relies on $\mathcal{E}(\delta)$ taking place, which occurs with high probability (i.e., by union bound, $\mathbb{P}(\mathcal{E}(\delta)) \geq 1-2\bar N\delta$), 
there is still a nonzero probability $\mathcal{E}(\delta)$ will not occur.
In such cases, we have the following solution, which is well-suited for the online (i.e., single trajectory) setting, where one cannot reset the system. 
Let $\hat{t}$ be the first time $\mathcal{E}_{\hat{t}(\tau+1)}(\delta)$ does not occur.
Then we pretend to reset the system with a new initial state $x_0 = x_{\hat{t}(\tau+1)}$, i.e., use the same trajectory but re-index time so that $\hat{t}(\tau+1)$ is the new starting time and adjust $\mathcal{E}_t(\delta)$ to account for this new initial state.
Although the new initial state $x_{\hat{t}(\tau+1)}$ may be large,~\cref{prop:pqgp_a9} says $x_{\hat{t}(\tau+1)}$ will almost surely be bounded (its norm will be influenced by the covariance matrix from~\eqref{eq:a15_LQR}).

Equipped with these properties of the state-action pair,
we will first show the stochastic estimates $\tilde{H}_t$ and $\tilde{b}_t$ also exhibit a ``light-tail''  property.
Recall the feedback lies in $u_t \in \mathbb{R}^m$ and that $\sigma > 0$ is the added policy disturbance level in~\eqref{eq:randomized_linear_state_feedback}.
\begin{lemma} \label{lem:a1_LQR}
    Suppose $K \in \mathcal S$ is the current policy and the events $\mathcal{E}_t(\delta)$ and $\mathcal{E}_{t+1}(\delta)$ occur for some $\delta \in (0,1/e]$.
    Then there exists constants
    \begin{align*}
        M_H
        &:=
        2 + 64(1+\|K\|)^4
        [O(1) {J(K)}/{\sigma_{\min}(Q)} + \|x_0\|^2 + \sigma^2 \cdot (m + 1) ]^2, \\
        M_b &:=
        M_H \max\{\|Q\|,\|R\|\},
    \end{align*}
    such that $\|\tilde{H}_t\| \leq M_H \big(\ln \frac{1}{\delta}\big)^{2}$ and $\|\tilde{b}_t\| \leq M_b \big(\ln \frac{1}{\delta}\big)^{2}$,
    where $O(1)$ is some absolute constant.

\end{lemma}
A proof can be found in~\cref{sec:b7_LQR}. 

We next bound the expected value $H = \mathbb{E}[\tilde H]$,
%
where can be shown similarly to~\cite[Lemma B.2]{yang2019provably}.
\begin{lemma} \label{lem:a6_OPN}
    Suppose $K \in \mathcal S$. Then matrix $H=\mathbb{E}_{\Pi_K}[\tilde{H}]$ has norm
    $    \|H\| 
        \leq L_H$
        where we define $L_H := 
        1 + 4(1 + \|K\|^2_F)\frac{J(K)}{\sigma_{\min}(Q)} + \sigma^2 \cdot (m+2)$.
\end{lemma}

Next, we show the mean noise can be controlled in the sense that the bias of the stochastic estimate $\tilde{H}_t$ and $\tilde{b}_t$ can be made arbitrarily small.
Denote $\TSig_K := \lim_{t \to \infty} \TSig_K^{(t)}$, where the latter matrix is from~\eqref{eq:a15_LQR} (the limit exists whenever $K \in \mathcal S$, c.f.~\cite[Lemma 18]{fazel2018global}).
\begin{lemma} \label{lem:lqgp_a22}
    Use the same assumptions and constants as~\cref{lem:a1_LQR}.
    Letting $\rho \equiv \rho(A-BK) < 1$ and defining
    \begin{align*}
        C
        &:=
        [ M_H^{1/4} + \sqrt{(n+m)/(1-\rho(A-BK))} ]/2 \\
        O_H 
        &:= 
        65 (n+m) \max\{\|\TSig_K\|^2, \|\TSig_K\|^4\} \\
        O_b
        &:= 
        41 \max\{\|Q\|,\|R\|\}(n+m)^2 \max\{\|\TSig_K\|^2, \|\TSig_K\|^4\},
    \end{align*}
    then we have for any mixing time $\tau \geq 1$,
    \begin{align*}
        &\|H - \mathbb{E}[\tilde H_{t+\tau} \vert \mathcal{F}_{t-1}]\| \leq CM_H \big(\ln \frac{1}{\delta} \big)^{3/2} \rho^\tau + O_H \sqrt{\delta} \\
        &\|b - \mathbb{E}[b_{t+\tau} \vert \mathcal{F}_{t-1}]\| \leq CM_b  \big(\ln \frac{1}{\delta}\big)^{3/2} \rho^\tau + O_b \sqrt{\delta}.
    \end{align*}
\end{lemma}
The proof is in~\cref{sec:b7_LQR}. 
This lemma illustrates that choosing a sufficiently large mixing time $\tau$ can help reduce the bias, where the bias is affected by the correlation between samples $\xi_t$. 
See~\cref{sec:general_pd} for more details on the mixing time.

To measure progress towards minimizing $f$ from~\cref{eq:a10_OPN}, we define the 
$Q$-gap function (where with some abuse of notation, $z \equiv (\vartheta, y)$ and $\tilde{z} \equiv (\tilde{\vartheta},\tilde{y})$)
is defined as $Q(\tilde{z}, z) :=  \langle H\tilde \vartheta - b,y \rangle  - \langle H\vartheta - b, \tilde y \rangle$,
and the primal-dual gap function is 
\begin{align} \label{eq:primal_dual_gap_func}
    \textstyle g(\tilde{z}) := \max_{z  \in X \times Y} Q(\tilde{z},z).
\end{align}
If there is a $\vartheta^*$ such that $H\vartheta^*=b$, then
$g(\tilde z) \geq f(\tilde{\vartheta})$ for all $\tilde y \in Y$. In this case, minimizing~\eqref{eq:primal_dual_gap_func} also  minimizes $f$ from~\cref{eq:a10_OPN}.
In view of this definition, we show the primal-dual gap function satisfies a regularity condition related to sharpness, as first introduced by Polyak~\cite{polyak1979sharp}.
This condition also appears in some data science problems~\cite{davis2023stochastic}. We will use it to help derive faster convergence rates.
Recall $\sigma \geq 0$ is the added policy disturbance  from~\eqref{eq:randomized_linear_state_feedback}, which can be arbitrarily set.

\begin{lemma}
\label{lem:lqgp_a29}
    Let $\rho \equiv \rho(A-BK)$. If $K \in \mathcal S$ and $\sigma^2 \geq \sigma_{\min}(\Psi)[1 + \|K\|^2]$, then
    \begin{align*}
        \mu := 
        \big( \sqrt{n+m}(1 + [\sigma_{\min}(\Psi)]^{-2})/(1-\rho^2) \big)^{-1}
        > 0
    \end{align*}
    satisfies $\mu \|\vartheta-\vartheta^*\| \leq g(z)$ for all $z \equiv [\vartheta,y] \in X \times Y$. 
\end{lemma}
The proof is given in~\cref{sec:b7_LQR}. 

We will now utilize these properties to derive a primal-dual method to solve the policy evaluation problem.

\subsection{A primal-dual algorithm for policy evaluation under Markovian noise} \label{sec:general_pd}
We use a primal-dual method,~\cref{alg:a1_OPN}, to find a near optimal primal-dual solution $z \equiv [\vartheta, y] \in X \times Y$ for the min-max problem~\eqref{eq:a10_OPN}. In each iteration, the algorithm separately collects $\tau$ samples to estimate gradients for the primal and dual steps. 
We refer to $\tau \geq 1$, which can be arbitrarily set, as the \textit{mixing time}, since its effect on convergence is related to the mixing time of ergodic Markov chains.
We call this method the conditional stochastic primal-dual method since $\tau$ must satisfy a condition (\cref{prop:a2_OPN}) for the final solution to attain a desired accuracy~\cite{kotsalis2022simple}. 
With some abuse of notation, we write the $t$-th stochastic estimates for the dual step as $\tilde{H}_{t,Y}$ and $\tilde{b}_{t,Y}$ and for primal step, $\tilde{H}_{t,X}$ and $\tilde{b}_{t,X}$ (see Lines~\ref{line:est_dual} and~\ref{line:est_prim}). 
The reason for collecting $\tau$ samples to estimate one gradient is to reduce the bias (e.g., correlation across time) between the primal and dual steps.
After completing $k$ primal and dual steps, we return a weighted average of all past $k$ points.

\begin{algorithm}
\caption{Conditional stochastic primal-dual (CSPD)}
\label{alg:a1_OPN}
\begin{algorithmic}[1]
\STATE{Input: $X,Y \subseteq \mathbb{R}^N$; $\vartheta_{-1}=\vartheta_0 \in X$, $y_0 \in Y$; $k, \tau \in \mathbb{Z}_{++}$; $\{\eta_t, \lambda_t, \zeta_t\}_{t=1}^k$}
\FOR{$t=1,\ldots,k$}
\STATE{$g_{t} = \vartheta_{t-1} + \zeta_t(\vartheta_{t-1} - \vartheta_{t-2})$}
\STATE{Collect $\tau$ samples, use $\tau$-th sample to set $(\tilde{H}_{t,Y},\tilde{b}_{t,Y}) = (\tilde{H}_{(2t-1)\tau)}, \tilde{b}_{(2t-1)\tau)})$} \label{line:est_dual}
\STATE{$y_{t} = \mathrm{argmin}_{y  \in Y} \ \{ \langle  -\tilde{H}_{t,Y}g_{t} + \tilde{b}_{t,Y}, y \rangle + \frac{\lambda_t}{2}\|y-y_{t-1}\|^2 \}$} 
\STATE{Collect $\tau$ more samples, use $\tau$-th sample to set $\tilde{H}_{t,X} = \tilde{H}_{2t\tau}$} \label{line:est_prim}
\STATE{$\vartheta_{t} = \mathrm{argmin}_{\vartheta \in X} \ \{ \langle y_{t}, \tilde{H}_{t,X}\vartheta \rangle + \frac{\eta_t}{2} \|\vartheta - \vartheta_{t-1}\|^2 \}$} 
\ENDFOR
\STATE{$(\bar{\vartheta}_k,\bar{y}_k) = \frac{2}{k(k+1)} \sum\limits_{t=1}^k t \cdot (\vartheta_t,y_t)$}
\RETURN $\bar{\vartheta}_k$
\end{algorithmic}
\end{algorithm}

We provide the general convergence result.
\begin{proposition} \label{prop:a1_OPN}
    Let $\{\gamma_t,\eta_t,\lambda_t,\zeta_t\}_t$ be a set of non-negative reals satisfying
        $\gamma_{t-1} \eta_{t-1} \leq \gamma_{t}\eta_t$, $\gamma_{t-1} \lambda_{t-1} \leq \gamma_{t}\lambda_t$, and $\gamma_{t}\zeta_t = \gamma_{t-1}$,
    and let there exist some $p,q \in (0,1)$ satisfying $\|H\|^2/(p \lambda_{t}) - q\eta_{t} \leq 0$ for all $t$. Then
    \begin{align*}
        \sum\limits_{t=1}^k \gamma_t Q(z_t,z) 
        \leq 
        \gamma_k\eta_{k}D_X^2 + \gamma_k\lambda_{k}D_Y^2 +  \sum\limits_{t=1}^{k} \Lambda_{t}(\vartheta,y),
    \end{align*}
    where $z_t = [\vartheta_t, y_t]$, $D_X^2 := \max_{x,x' \in X} \|x-x'\|^2$ (and similarly for $D_Y^2$) and
    \begin{align*}
        \Lambda_t(\vartheta,y) 
        &:= -\frac{\gamma_{t}\lambda_t}{2}(1-p) \|y_{t}-y_{t-1}\|^2 -\frac{(1-q)\gamma_{t}\eta_t}{2}\|\vartheta_{t}-\vartheta_{t-1}\|^2 \\
        &\hspace{10pt}
        + \gamma_{t}\langle  (\tilde{H}_{t,Y} - H) g_{t} - (\tilde{b}_{t,Y} - b), y_{t}-y \rangle - \gamma_{t}\langle(\tilde{H}_{t,X} - H)^T y_t, \vartheta_{t}-\vartheta \rangle.
    \end{align*}
\end{proposition}
A proof can be found within~\cite[Theorem 3.8]{lan2020first}. Next, we bound the stochastic error on a convex set $U$ (where $U$ can either be $X$ or $Y$). Due to space constraints, we only show the main steps of the proof and provide references to auxiliary parts.
\begin{proposition}  \label{prop:a3_OPN}
    Let $\{\gamma_t, c_t \}$ be non-negative scalars satisfying $\gamma_{t-1} c_{t-1} \leq \gamma_{t}c_t$. Let $\tilde{G}_{t\tau}$ be an $\mathcal{F}_{t\tau}$-measurable stochastic vector such that
    \begin{align}
        &\|\tilde{G}_{t\tau} - \mathbb{E}[\tilde{G}_{t\tau} \vert \mathcal{F}_{(t-1)\tau} ]\| \leq \varsigma \label{eq:a5_OPN} \text{ w.p. $1-\delta$}, \\
        &\|G - \mathbb{E}[\tilde{G}_{t\tau} \vert \mathcal{F}_{(t-1)\tau}]\|_* \leq CM\rho^\tau  + O\sqrt{\delta} \label{eq:a6_OPN},
    \end{align}
    where $C,M,\varsigma,\tau \in \mathbb{R}_+$, $\rho,\delta \in [0,1)$, and $G \in \mathbb{R}^n$ are constants. Then for any vectors $\{u_t,v_t \in U \}_{t}$ where $u_t$ and $v_t$ are $\mathcal{F}_{t \tau}$-measurable, with probability at least $1-(k+1)\delta$,
    \begin{align*}
        &\sum\limits_{t=1}^{k} \gamma_{t} \langle G - \tilde{G}_{t \tau}, u_{t}-u \rangle  - \frac{c_t \gamma_t}{2}\|u_{t}-u_{t-1}\|^2 \leq 
        \gamma_k c_{k}D_U^2\\
        &
        + \sum\limits_{t=1}^{k} \gamma_{t}\bigg(  (C M \rho^\tau + O\sqrt{\delta})D_U+ \frac{4((CM\rho^{\tau})^2 + O^2\delta + \varsigma^2)}{c_t} \bigg) 
        + \varsigma D_U\sqrt{ 8 \ln({1}/{\delta})\sum\limits_{t=1}^{k} \gamma_{t}^2},
    \end{align*}
    where $D_U^2 = \max_{\upsilon,\upsilon' \in U} \|\upsilon-\upsilon'\|^2$.
\end{proposition}
\begin{proof}
    With $u_0^v := u_0$ and $\Delta_t := G - \tilde{G}_{t \tau}$, let $u^v_{t} := \mathrm{argmin}_{u \in U} \{\langle \Delta_t, u \rangle + \frac{c_t}{2} \|u-u^v_{t-1}\|^2 \}$.
    Since $G$ is a constant and $\tilde{G}_{t \tau}$ is assumed to be $\mathcal{F}_{t \tau}$-measurable, a simple proof by induction ensures $u_t^v$ is $\mathcal{F}_{t \tau}$-measurable. Now, we decompose 
    \begin{align*}
        &\sum\limits_{t=1}^{k} \gamma_{t} \langle \Delta_t, u_{t}-u \rangle - \frac{c_t \gamma_t}{2}\|u_{t}-u_{t-1}\|^2  \\
        &= 
        \sum\limits_{t=1}^{k} \underbrace{\gamma_{t} \langle \Delta_t, u_{t}-u_{t-1} \rangle 
        - \frac{c_t \gamma_t}{2}\|u_{t}-u_{t-1}\|^2}_{A_t}  
        + \underbrace{\gamma_{t} \langle \Delta_t, u_{t-1}-u_{t-1}^v \rangle}_{B_t} + \underbrace{\gamma_{t} \langle \Delta_t, u_{t-1}^v - u \rangle}_{C_t}.
    \end{align*}
    We have the following bound, which holds with probability $1-k\delta$ by union bound:
    \begin{align*}
        \sum\limits_{t=1}^k A_t  
        &\leq 
        \sum\limits_{t=1}^k \frac{\gamma_{t}}{2c_t} \|G - \tilde{G}_{t \tau}\|_*^2 \\
        &\leq
        \sum\limits_{t=1}^k \frac{2\gamma_{t}}{c_t} 
        [\|G - \mathbb{E}[\tilde{G}_{t\tau} \vert \mathcal{F}_{(t-1)\tau}]\|_*^2 + \|\mathbb{E}[\tilde{G}_{t\tau} \vert \mathcal{F}_{(t-1)\tau}] -  \tilde{G}_{t \tau}\|_*^2] \\
        &\leq
        \sum\limits_{t=1}^k \frac{2\gamma_{t}}{c_t} 
        ((CM\rho^\tau)^2 + O^2\delta + \varsigma^2 ),
    \end{align*}
    which is due to Young's inequality, $(a+b)^2 \leq 2a^2 + 2b^2$, and then~\eqref{eq:a5_OPN} and~\eqref{eq:a6_OPN}. Next, we can bound with probability $1-\delta$,
    \begin{equation*} 
    \begin{split}
        \sum\limits_{t=1}^k B_t 
        &= 
        \sum\limits_{t=1}^k
        (B_t - \mathbb{E}[B_t \vert \mathcal{F}_{(t-1)\tau}]) + \mathbb{E}[B_t \vert \mathcal{F}_{(t-1)\tau} ] \\
        &\leq
        \varsigma D_U \Big[ 8\ln(1/{\delta}) \sum\limits_{t=1}^{k} \gamma_{t}^2 \Big]^{-1/2} + [CM\rho^\tau + O\sqrt{\delta}]D_U \sum\limits_{t=1}^k \gamma_t,  
    \end{split}
    \end{equation*}
    where we used the Azuma-Hoeffding inequality to bound the martingale-difference sequence $B_t - \mathbb{E}[B_t \vert \mathcal{F}_{(t-1)\tau}]$ (and used the fact $\Delta_t$ is $\mathcal{F}_{t\tau}$-measurable while $u_{t-1}-u_{t-1}^v$ are $\mathcal{F}_{(t-1)\tau}$-measurable as well as~\eqref{eq:a5_OPN} and definition of $D_U$; see~\cite[Lemma 7]{wang2017finite} for full details) and~\eqref{eq:a6_OPN} to bound $\mathbb{E}[B_t \vert \mathcal{F}_{(t-1)\tau}]$. 
    Finally, using $\gamma_{t-1}c_{t-1} \leq \gamma_tc_t$, we get
        $\sum\limits_{t=1}^k C_t 
        \leq
        \gamma_k c_k D_U^2 + \sum\limits_{t=1}^k \frac{\gamma_t}{c_t} \|\Delta_t\|_*^2
        \leq
        \gamma_k c_k D_U^2 + \sum\limits_{t=1}^k \frac{2\gamma_t}{c_t}((CM\rho^\tau)^2 + O^2\delta + \varsigma^2)$, 
    where the first and second inequality can be shown similarly to~\cite[Lemma 4.10]{lan2020first} and the bound on $A_t$, respectively. Combining everything finishes the proof.
\end{proof}

We are now ready to complete the proof of convergence. Recall the gap function $g(\bar{z}) \equiv \max_{z \in X \times Y}Q(\bar{z},z)$ and the weighted average solution $\bar{z}_k \equiv [\bar{x}_k,\bar{y}_k]$. 
\begin{theorem} \label{thm:a1_OPN}
    Suppose $K \in \mathcal S$ and $\mathcal{E}(\delta)$ take place for some $\delta \in (0,1/e]$.
    Under the same assumptions as~\cref{prop:a1_OPN}, then with probability $1-2(k+1)\delta$,
    \begin{equation*} 
    \begin{split}
        \big( &\sum\limits_{t=1}^{k} \gamma_{t} \big) g(\bar{z}_k)  
        \leq
        2\gamma_k(\eta_{k}D_X^2 + \lambda_{k}D_Y^2) + 2(D_XM_X + D_Y M_Y)\sqrt{ 8\ln\frac{1}{\delta}\sum\limits_{t=1}^{k} \gamma_{t}^2 }  \\
        &+\sum\limits_{t=1}^{k} \gamma_{t} \bigg[   C(D_XM_X + D_YM_Y) \rho^\tau + (O_XD_X + O_YD_Y)\sqrt{\delta} + \frac{16M_X^2}{\eta_t(1-q)} + \frac{16M_Y^2}{\lambda_t(1-p)}  \\
        &\hspace{40pt} + 4C^2\Big( \frac{M_X^2}{\eta_t(1-q)} + \frac{M_Y^2}{\lambda_t(1-p)}\Big)\rho^{2\tau}  + 4 \Big( \frac{O_X^2}{\eta_t(1-q)} + \frac{O_Y^2}{\lambda_t(1-p)} \Big) \delta \bigg],
    \end{split}
    \end{equation*}
    where $\overline{\zeta} := \max_t \zeta_t$, $D_U^2 := \max_{u,u' \in U} \|u-u'\|^2$ for some space $U$, and
    \begin{align*}
        \Omega_Y &:= \max\{\|y\| : y \in Y\}, 
        & \Omega_X &:= \|\vartheta^*\| + (1+\overline{\zeta})\sqrt{2}D_X, \\ 
        M_X &:= M_H \Omega_Y \Big(\ln\frac{1}{\delta}\Big)^2, 
        & M_Y &:= \big( M_H\Omega_X + M_b \big) \Big(\ln \frac{1}{\delta}\Big)^2, \\
        O_X &:= O_H\Omega_Y,  
        & O_Y &:= O_H\Omega_X + O_b.
   \end{align*}
   Constants $M_H, M_b, C, O_H, O_b$ are from~\cref{lem:a1_LQR} and~\cref{lem:lqgp_a22}. 
\end{theorem}
\begin{proof}
    We first decompose the error $\sum\limits_{t=1}^{k} \Lambda_t(\vartheta,y)$ defined in~\cref{prop:a1_OPN} as
    \begin{align*} 
        \sum\limits_{t=1}^{k} \Lambda_t(\vartheta,y)  
        &= 
        \sum\limits_{t=1}^{k}  \gamma_{t}\langle  (\underbrace{\tilde{H}_{t,Y}g_t - \tilde{b}_{t,Y}}_{\tilde{G}_{t,Y}^\tau}) - (\underbrace{Hg_t - b}_{G_Y}), y_t - y \rangle -\frac{\gamma_{t}\lambda_t(1-p)}{2} \|y_{t}-y_{t-1}\|^2 \\
        &\hspace{25pt}
         - \gamma_{t}\langle {\tilde{H}_{t,X}^T y_t} - {H^Ty_t}, \vartheta_{t}-\vartheta \rangle  -\frac{(1-q)\gamma_{t}\eta_t}{2}\|\vartheta_{t}-\vartheta_{t-1}\|^2.
    \end{align*}
    We wish to use~\cref{prop:a3_OPN} to bound the first two summands (the last two summands can be similarly shown), which requires verification of~\eqref{eq:a5_OPN} and~\eqref{eq:a6_OPN}. First, because $\tilde{H}_{t,Y}$ and $\tilde{b}_{t,Y}$ are generated with the $((2t-1)\tau)$-th sample in~\cref{alg:a1_OPN} and $g_t = \vartheta_{t-1} + \zeta_t (\vartheta_{t-1}-\vartheta_{t-2})$, then $\tilde{G}^{\tau}_{t,Y}$ is $\mathcal{F}_{(2t-1)\tau}$-measurable. Denoting $\mathcal{F}' = \mathcal{F}_{(2t-2)\tau}$, then with probability $1-\delta$,
    \begin{align*}
        &\|\tilde{G}_{t,Y}^\tau - \mathbb{E}[\tilde{G}_{t,Y} \vert \mathcal{F}'] \|_* \\
        &\leq
        [\|\tilde{H}_{t,Y}\|_* + \|\mathbb{E}[\tilde{H}_{t,Y} \vert \mathcal{F}']\|](\|\vartheta^*\| + \|\vartheta_{t-1}-\vartheta^*\| + \zeta_t\|\vartheta_{t-1}-\vartheta_{t-2}\|) \\
        &\hspace{10pt} 
        + \|\tilde{b}_{t,Y}\|_* + \|\mathbb{E}[\tilde{b}_{t,Y} \vert \mathcal{F}']\|_* \\
        &\leq
        [\|\tilde{H}_{t,Y}\|_* + \mathbb{E}[\|\tilde{H}_{t,Y} \| \vert \mathcal{F}']](\|\vartheta^*\| + (1+\zeta_t)D_X) + \|\tilde{b}_{t,Y}\|_* + \mathbb{E}[\|\tilde{b}_{t,Y} \|_* \vert \mathcal{F}']\\
        &\leq
        2M_H\Big(\ln\frac{1}{\delta}\Big)^2 ( \|\vartheta^*\| + (1+\zeta_t)\sqrt{2}D_X ) 
        + 2M_b\Big(\ln \frac{1}{\delta}\Big)^2
        \leq
        2M_Y,
    \end{align*}
    where the second inequality is by definition of $D_X^2$ and Jensen's inequality, and the last line used~\cref{lem:a1_LQR} (recall we assumed $K \in \mathcal S$ and $\mathcal{E}(\delta)$ occur). 
    This verifies~\eqref{eq:a5_OPN} with $\varsigma = 2M_Y$ up to a shift in the index $t\tau$. One can similarly show
    $
        \|{G}_{Y} - \mathbb{E}[\tilde{G}_{t,Y}^\tau \vert \mathcal{F}_{(2t-2)(\tau+1)}] \|_* 
        \leq
        CM_Y \rho^\tau + O_Y \sqrt{\delta}
    $ using~\cref{lem:lqgp_a22}, verifying~\eqref{eq:a6_OPN}.
    
    Fixing $U = Y$, $u_t = y_t$, $u = y$, and $c_t = \lambda_t$ so that $\gamma_{t-1}c_{t-1} \leq \gamma_tc_t$ (from the assumption in~\cref{prop:a1_OPN}), we can use~\cref{prop:a3_OPN} to bound the first two summation in the definition of $\Lambda_t$.
    After applying a similar bound to the latter two summations in $\sum\limits_{t=1}^{k} \Lambda_t$,
    we plug the resulting bound back in~\cref{prop:a1_OPN}, which bounds the primal-dual gap.
    The proof is complete by invoking convexity of $Q$
    and dividing through by the sum over $\gamma_t$.
\end{proof}

We can now specify an explicit step size to obtain the following convergence rate. 
The proof is similar to~\cite[Corollary 4.3]{lan2020first}, so we skip it.
\begin{corollary} \label{cor:a1_OPN}
    Let $\eta_t = \frac{3\sqrt{2}\|H\|D_Yt + 6M_Xt^{3/2}}{2\sqrt{2}D_Xt}$, $\lambda_t=\frac{3\sqrt{2}\|H\|D_Xt + 6M_Yt^{3/2}}{2\sqrt{2}D_Yt}$, and $\zeta_t=\frac{t-1}{t}$, where $M_I$ and $O_I$ for $I \in \{X,Y\}$ are defined in~\cref{thm:a1_OPN}. If $K \in \mathcal S$ and $\mathcal{E}(\delta)$ take place for some $\delta \in (0,1/e]$, then with probability $1-2(k+1)\delta$,
    \begin{align*}
        g(\bar{z}_k)
        &\leq 
        \frac{12\|H\|D_XD_Y}{k+1} + \frac{2(48 + 3\sqrt{2} + \frac{16\sqrt{2}}{\sqrt{3}} \sqrt{\ln \frac{1}{\delta} })(D_XM_X + D_YM_Y)}{\sqrt{k}} \\
        &\hspace{10pt} + \frac{24}{\sqrt k} \Big[ C^2( D_XM_X + D_YM_Y )\rho^{2\tau} + \big( \frac{D_XO_X^2}{2M_X} + \frac{D_YO_Y^2}{2M_Y} \big)\delta \Big] \\
        &\hspace{10pt} + C \big( D_XM_X + D_YM_Y \big)\rho^\tau + (D_XO_X + D_YO_Y)\sqrt{\delta}.
    \end{align*}
\end{corollary}

In view of~\cref{lem:lqgp_a29} and assuming the mixing time $\tau$ and added noise $\sigma$ are sufficiently large, the above result tells us $\bar{z}_k$ satisfies $\|\bar{z}_k-z^*\|^2 \leq O(\frac{\|H\|^2}{\mu^2 k^2} + \frac{\mathrm{var}}{\mu^2 k})$ with high probability, where $z^* \equiv [\vartheta^*; y^*]$ is the optimal primal-dual solution to~\eqref{eq:a10_OPN}, ``$\mathrm{var}$'' bounds stochastic estimation errors, and $k$ is the number of iterations. 
It turns out one can further modify the algorithm to decrease the overall error faster.

To do so, we introduce a shrinking multi-epoch algorithm in~\cref{alg:a2_OPN}. It repeatedly calls~\cref{alg:a1_OPN} and uses the solution to the previous call to warm-start the next call. We name each call an \textit{epoch}. Each epoch $s$ will use an updated feasible set for the primal variable $X_s$ that is shrunk from the previous epoch. The feasible set for the dual variable $Y$ remains unchanged. A key feature is that the optimal solution to~\eqref{eq:a10_OPN} is still feasible every epoch, i.e., $\vartheta^* \in X_s$. Since $X_s$ gets smaller, we also update the diameter $D_{X_s}$. The choice of step sizes still adheres to those prescribed in~\cref{cor:a1_OPN} up to the difference in the diameter $D_{X_s}$, while the number of iterations $k_s$ increases each epoch. 

\begin{algorithm}
\caption{Shrinking multi-epoch CSPD}
\label{alg:a2_OPN}
\begin{algorithmic}
\STATE{Input: $X \subseteq \mathbb{R}^n$; $p_0 \in X$; $D_0 \in \mathbb{R}_{++}$; $k, \tau, S \in \mathbb{Z}_{++}$}
\FOR{$s=1,\ldots,S$}
\STATE{$D_s^2 := 2^{-(s-1)} D_0^2$}
\STATE{$X_s := \{\vartheta \in X : \|p_{s-1}- \vartheta\|^2 \leq D_s^2 \}$}
\STATE{$p_{s} \gets$~\cref{alg:a1_OPN} with $X_s$, $p_{s-1}$, $k_s$, $\tau_s$, $\delta_s$, $\{\eta_t,\lambda_t,\zeta_t\}_t$ specified by~\cref{cor:a1_OPN,prop:a2_OPN}}
\ENDFOR
\RETURN $p_S$
\end{algorithmic}
\end{algorithm}

We now establish the main convergence result, similar to~\cite[Lemma 4.5]{lan2020first}. 
\begin{proposition} \label{prop:a2_OPN}
    Consider the same assumptions and parameters as~\cref{cor:a1_OPN}. Suppose $\|p_0-\vartheta^*\|^2 \leq D_0^2$ and $\sigma$ satisfies the bound from~\cref{lem:lqgp_a29}.
    Set
    \begin{align*}
        k_s 
        &=
        \Big \lceil 400 \max \Big\{  \frac{\|H\| D_Y}{\mu}, \frac{4000 + 256\ln (1/\delta)}{\mu^2} \big( M_X^2 + \frac{D_Y^2M_Y^2}{D_0^2} \cdot 2^{s} \big) \Big\} \Big \rceil \\
        \tau_s 
        &\equiv \tau 
        =
        \Big \lceil \frac{\max\big\{ \frac{1}{2} \ln(\frac{228C^2}{\mu\sqrt{\epsilon}} (D_XM_X + D_YM_Y), \ln(\frac{72C}{\mu\sqrt{\epsilon}}(D_XM_X + D_YM_Y) \big\}}{\ln(1/\rho)} \Big \rceil \\
        \delta_s
        &\equiv \delta
        \leq
        \min\big\{ \sqrt{\frac{\epsilon \mu^2}{18}}\big(\frac{D_XO_X^2}{2M_X} + \frac{D_YO_Y^2}{2M_Y}\big), \frac{\epsilon}{18\mu}(D_XO_X + D_YO_Y)^2 \big\}.
    \end{align*}
    Choosing the number of epochs as $S = \big \lceil \log_2 \frac{D_0^2}{\epsilon} \big \rceil$, then $\|p_S- \vartheta^*\|^2 \leq \epsilon$ with probability at least $1-2(N_S + S)\delta$, and the total iterations is at most
    \begin{align*}
        N_S 
        := 
        O (1) \Big (\frac{ \|H\|D_Y }{\mu} \ln\big(\frac{D_0}{\epsilon}\big) 
        + \frac{1 + \ln(1/\delta)}{\mu^2} \big( M_X^2 \ln\Big(\frac{D_0}{\epsilon}\Big) + \frac{D_Y^2 M_{Y}^2}{\epsilon}\big) \Big), 
    \end{align*}
    while the sample complexity is $2N_S\tau$. Here, $O(1)$ is some absolute constant. 
\end{proposition}
{\begin{proof}
    We will show by mathematical induction that for all $s=0,\ldots,S$, $\|p_s-\vartheta^*\| \leq 2^{-s} D_0^2$ with probability $1-\sum\limits_{i=1}^s 2(k_i +1)\delta$.
    The base case of $s=0$ is true by the assumption on $p_0$. 
    
    Now suppose the claim is true for $s-1$.
    By the inductive hypothesis, $\|p_{s-1}-\vartheta^*\|^2 \leq 2^{-s+1}D_0^2$ with probability $1-\sum\limits_{i=1}^{s-1}2(k_i+1) \delta$. Therefore, the optimal solution is still in the feasible region, i.e., $\vartheta^* \in X_s$. 
    Now, conditioned on the success of the previous epochs,~\cref{cor:a1_OPN} combined with our choice in parameters $k_s$, $\tau_s$, and $\delta_s$ guarantees the primal-dual gap function respects with probability $1-2(k_s+1)\delta$,
    $g([p_s; y_s]) \leq \mu \sqrt{2^{-s} D_0^2}$,
    where $[p_s,y_s]$ is primal-dual solution from the $s$-th call to~\cref{alg:a2_OPN}.
    Squaring both sides and using~\cref{lem:lqgp_a29} (recall we assumed $K \in \mathcal S$ from~\cref{cor:a1_OPN}), we conclude $\mu^2\|p_s-\vartheta^*\|^2 \leq g([p_s; y_s])^2 \leq \mu^2 \cdot 2^{-s} D_0^2$.
    By union bound, this holds unconditionally with probability $1-\sum\limits_{i=1}^s 2(k_s+1)\delta$, and this completes the proof by induction.
    
    Finally, the total number of iterations can be derived by summing $k_s$ across all epochs $s$.

\end{proof}
}

The above result improves upon~\cref{cor:a1_OPN} by reducing the error term corresponding to $\|H\|$ (i.e., deterministic error) without worsening the stochastic estimation errors.
Next, we apply this generic primal-dual method to the setting of estimating the natural gradient $E_{K_t}$ within the natural policy gradient method.

\subsection{Combining natural policy gradient with the primal-dual method to solve LQR} \label{sec:a2_LQR}
\hide{}{Finally, let us recall that in policy evaluation, we use the primal-dual method of~\cref{alg:a2_OPN} to minimize $\|x-x^*\|_2^2$ by solving 
\begin{align*}
    \min_{x \in X} \|Ax-b\|_2 = \min_{x \in X}\max_{y : \|y\|_2 \leq 1} \langle y, Ax-b \rangle.
\end{align*}
In view of the general min-max problem~\eqref{eq:a10_OPN}, we have $Y := \{y : \|y\|_2 \leq 1\}$, so 
\begin{align*}
    \Omega_Y &\equiv \max_{y \in Y} \|y\| = 1\\
    D_Y &\equiv \max_{y,y' \in Y} \frac{1}{2}\|y-y'\|^2 \leq 2.
\end{align*}
Recall that constant $\Omega_Y$ appears in the definition of the gradient and variance bound $M_X$ and $\varsigma_X$ for the primal player in~\eqref{eq:a8_OPN}, while $D_Y$ appears in the final iteration and sampling complexity of~\cref{cor:a2_OPN}. 

We are now ready to provide the sampling complexity from all the multiplication invocations of policy evaluation~\cref{alg:a2_OPN}) from policy optimization~\cref{alg:b1_LQR}). First, let us recall the constants $C_1$, $C_2$, and $C_3$ from~\eqref{eq:b1_LQR}, which are used for policy optimization. If we view all norms and smallest singular values on the system parameters $A$, $B$, $Q$, $R$, as well as the noise $\Psi$, $\Psi_\sigma$, and $\sigma^2$ as constants, we can bound the size of these constants as 
\begin{align*}
    C_1 
    &= O(1)J(K_0) \\
    C_2 &= 
    O(1) \\ 
    C_3 &= O(1)J(K_0). 
\end{align*}
In view of these constants, we know that we can ensure the stability of all controllers $K_t$ and $J(K_{N_{\text{out}}}) - J(K^*) \leq \varepsilon$ by setting the number of policy optimization iterations to be
\begin{align*}
    N_{\text{out}} 
    &:= 
    \lceil 8J(K^*)C_1/C_2 \rceil \cdot \lceil \log_2(J(K_0)/\varepsilon) \rceil = O(1)J(K^*)J(K_0) \ln(J(K_0)/\varepsilon), 
\end{align*}
while ensuring that each call to policy evaluation returns an $p_S$ such that $\|p_S-x^*\|^2 \leq \min\{C_4,C_5,C_6\}$, where
\begin{align*} 
    C_4
    &= 
    O(1),  \\
    C_5 
    &=
    O(1)\min \Big( \frac{1}{J(K_0)^2 J(K^*)}, \frac{1}{J(K_0)^3} \Big) \varepsilon,  \\
    C_6
    &=
    O(1)\Big( \frac{1}{J(K_0)^4} \varepsilon \Big)^{2/3}.
\end{align*}
In the complexity results below, we also view the norm of the initial state-action pair $X_0 \equiv (x_0,u_0)$, the controllers $K_t$, and the optimal cost $J(K^*)$ as constants. Moreover, we view any variable inside the log as a constant except for the failure rate $\delta$ and accuracy $\varepsilon$. In the following result, recall $u_t \in \mathbb{R}^k$. 
}
We can now prove $\tilde{O}(1/\varepsilon)$ samples suffice to solve LQR. 
First, define the radius $R_* = (\|Q\|_F + \|R\|_F) + (\|A\|_F^2 + \|B\|_F^2) \cdot \frac{ J(K_0)}{\sigma_{\min}(\Psi)}$. 


\begin{theorem} \label{thm:a3_OPN}
    Suppose $K_0 \in \mathcal S$ and $\varepsilon \in (0, J(K_0)]$, and set $\sigma^2 = \sigma_{\min}(\Psi) + 2\sigma_{\min}(R)^{-1}J(K_0)$ in~\eqref{eq:randomized_linear_state_feedback}.
    Run~\cref{alg:b1_LQR} with the  step size from~\cref{thm:b1_LQR} and $T=O(J(K_0)^2\ln(1/\varepsilon))$. 
    To compute the natural gradient via~\cref{alg:a2_OPN}, use the parameters from~\cref{prop:a2_OPN}, $D_0 = 2R_*$, $\epsilon = O(\varepsilon/J(K_0)^4)$, ${\delta} \in (0,1/e]$, and choose any initial solution $p_0$ satisfying $\|p_0\| \leq R_*$. 
    Then~\cref{alg:b1_LQR} outputs $K_T$ s.t. $J(K_T) - J(K^*) \leq \varepsilon$ with probability $1-8\bar{N} \delta$, and the total number of samples is 
    \begin{align*}
        \bar{N}
        &:= 
        O\Big( \frac{J(K_0)^{23} \cdot m^4(n+m) \cdot [\ln(1/{\delta}) + \ln(1/\varepsilon)]^{7} }{\varepsilon} \Big).
    \end{align*}
    Big-O hides dependence on absolute constants;  $A$, $B$, and $\Psi$ from~\eqref{dynamics}; $Q$ and $R$ from~\eqref{ergodic_cost}; and logarithmic dependence on $J(K_0)$, $n$, and $m$. 
\end{theorem}
\begin{proof}
    Our goal is to apply~\cref{thm:b1_LQR}. To do so, we need to verify the natural gradient estimation error is $\|E_{K_t} - E_{K_t}^\star\|_F^2 \leq \epsilon' := \min\{C_5, C_6, C_7\} = O(\varepsilon/J(K_0)^3)$, where the constants are defined in~\cref{thm:b1_LQR}. If this holds, then by our choice in $T$, ~\cref{thm:b1_LQR} guarantees the final iterate $K_T$ satisfies $J(K_T) - J(K^*) \leq \varepsilon$.

    To show the desired accuracy, recall from~\eqref{eq:natgrad},~\eqref{eq:phi_def}, and~\eqref{eq:linsys_bellman} the relation between the desired natural gradient $E_{K_t}$ and an approximate solution $\vartheta(K_t)$ for~\cref{eq:linsys_bellman} returned by~\cref{alg:a2_OPN} (denote $\vartheta^* \equiv \vartheta(K_t)^\star$ as the optimal solution):
    \begin{align*}
        \|E_{K_t} - E_{K_t}^\star\|^2_F
        &=
        \|\Theta(K_t)_{22}K_t - \Theta(K_t)_{21} - \Theta(K_t)^\star_{22}K_t - \Theta(K_t)^\star_{21}\|_F^2 \\
        &\leq
        2\|\Theta(K_t)_{22} - 
        \Theta(K_t)^\star_{22}\|_F^2\|K_t\|_F^2 + 2\|\Theta(K_t)_{21} - \Theta(K_t)^\star_{21}\|_F^2 \\
        &\leq
        2(\|K_t\|_F^2+1)\|\theta(K_t) - \theta(K_t)^\star\|^2 \\
        &\leq
        2(\|K_t\|_F^2+1)\|\vartheta(K_t) - \vartheta(K_t)^\star\|^2.
    \end{align*}
    Hence, we require accuracy $\|\vartheta(K_t) - \vartheta(K_t)^\star\|^2 \leq \epsilon := \epsilon'/(2(\|K_t\|_F^2+1))$.
    
    For us to use~\cref{prop:a2_OPN} to show $\|\vartheta(K_t) - \vartheta(K_t)^\star\|_2^2 \leq \epsilon$, we must verify its assumptions. Suppose event $\mathcal{E}({\delta})$ from~\eqref{eq:qgp_a57} occurs (we will bound the probability of failure at the end of the proof). 
    We know $K_t \in \mathcal S$ at iteration $t$ by~\cref{thm:b1_LQR}. 
    And by choice of $\sigma$, we can satisfy the assumption for~\cref{lem:lqgp_a29} because
    \begin{align*}
        \sigma^2
        &=
        \sigma_{\min}(\Psi)(1 + 2[\sigma_{\min}(\Psi)\sigma_{\min}(R)]^{-1}J(K_0)) \\
        &\geq
        \sigma_{\min}(\Psi)(1 + [\sigma_{\min}(\Psi)\sigma_{\min}(R)]^{-1}J(K_t)) 
        \geq
        \sigma_{\min}(\Psi)(1 + \|K_t\|_F^2),
    \end{align*}
    where the first inequality used $J(K_t) \leq 2J(K_0)$ from~\cref{thm:b1_LQR} and the second inequality applied~\cref{lem:lqgp_a13}.
    Finally, the choice in $D_0$ and assumption $\|p_0\| \leq R_*$ guarantees $\|p_0-\vartheta^*\|^2 \leq D_0^2$ ($\|\vartheta^*\|^2 \leq R_*^2$ can be shown by~\cite[Eqn 5.4]{yang2019provably} and~\cref{lem:lqgp_a13}). 
    So, we have verified all assumptions required by~\cref{prop:a2_OPN}.

    Having verified the assumptions for~\cref{prop:a2_OPN}, let us now bound $N_S$, $\tau$, and $\delta$ to achieve accuracy $\epsilon$. In view of $N_S$ from~\cref{prop:a2_OPN}, then combining the constants derived in~\cref{thm:a1_OPN},~\cref{lem:a1_LQR}, the assumption $\|\vartheta_0\| \leq R_* = O(J(K_0))$,~\cref{lem:a6_OPN},~\cref{lem:lqgp_a22},~\cref{lem:lqgp_a29}, choice of $\sigma^2$, and $D_{X_s} \leq D_0 = O(J(K_0))$ (since we shrink the feasible region), then
    $
        N_S 
        = 
        O\Big( \frac{\|K_t\|^8_F \cdot J(K_0)^{10}  \cdot m^4(n+m) (\ln(1/\delta))^{5}}{(1-\rho(A-BK_t))^2\epsilon} \Big).
    $
    \pedant{}{ 
        To see this (we will only show dependence on largest constants), recall from $N_S$ in~\cref{prop:a2_OPN} and the terms in~\cref{thm:a1_OPN},
        \begin{align*}
            N_S 
            &=         
            \frac{\|H\|D_Y}{\mu}\ln(\frac{1}{\epsilon}) + \frac{\ln(1/\delta)}{\mu^2}\Big(M_X^2\ln(\frac{D_0}{\epsilon}) + \frac{D_Y^2M_Y^2}{\epsilon} \Big) \\
            M_X 
            &= 
            M_H\Omega_Y(\ln\frac{1}{\delta})^\beta \\
            M_Y
            &=
            (M_H\Omega_X+M_b)(\ln\frac{1}{\delta})^\beta \\
            \Omega_X
            &=
            \|\vartheta^*\| + (1-\bar{\zeta})\sqrt{2}D_X
        \end{align*}
        where one can show $\Omega_Y = 1$, $D_Y = O(1)$, and $\bar{\zeta} \leq 2$.  \newline
        Now,~\cref{lem:a1_LQR} says $\beta = 2$ (along with $\|x_0\| = R_* = O(J(K))$ from~\cref{lem:lqgp_a30}),
        \begin{align*}
            M_H
            &=
            \|K\|^2(J(K) + \|x_0\|^2 + \sigma^2 \cdot m)^2
            =
            \|K\|^4 J(K_0)^4 m^2 \\
            M_b &= M_H =  \|K\|^4 J(K_0)^4 m^2 \\
            \Omega_X &= R_* = J(K_0).
        \end{align*}
        Thus,
        \begin{align*}
            M_X &= \|K\|^4 J(K_0)^4 m^2 (\ln \frac{1}{\delta})^2 \\
            M_Y &= \|K\|^4 J(K_0)^5 m^2 (\ln \frac{1}{\delta})^2.
        \end{align*}
        Meanwhile,~\cref{lem:a6_OPN} guarantees
        \begin{align*}
            \|H\| 
            \leq
            \|K\|_F^2J(K) + \sigma^2 \cdot m
            \leq
            (\|K\|_F^2 + m)J(K) 
            \leq
            m\|K\|_F^2J(K).
        \end{align*}
        On the other hand,~\cref{lem:lqgp_a29} says (recall we ignore constants like $\sigma_{\min}(\Psi)$)
        \begin{align*}
            \frac{1}{\mu}
            &\leq 
            O(1) \frac{\sqrt{n+m}}{1-\rho^2}.
        \end{align*}
        Bringing everything together, we have $\|H\| \leq M_X^2 \leq M_Y^2$, and so
        \begin{align*}
            N_S 
            &=         
            \frac{\ln(1/\delta)}{\mu^2} \cdot  \frac{D_Y^2M_Y^2}{\epsilon}\\
            &\leq
            \ln(1/\delta) \cdot \frac{(n+m)}{(1-\rho^2)^2} \cdot \|K\|^8 J(K_0)^{10} m^4 (\ln\frac{1}{\delta})^4 \cdot \frac{1}{\epsilon} \\
            &\leq
            \frac{\|K\|_F^{8} J(K_0)^{10} m^4 \cdot (n+m)}{(1-\rho^2)^2}  (\ln\frac{1}{\delta})^5 \cdot \frac{1}{\epsilon}
        \end{align*}
    }
    To bound $\tau$ from~\cref{prop:a2_OPN}, we use~\cref{cor:qgp_a1}, $\frac{1}{\log(1/u)} \leq \frac{1}{1-u}$ for $u \in (0,1)$, and~\cref{lem:lqgp_a22} to show
    $
        \tau = O\big( \frac{\ln(n/\epsilon)}{\log(1/\rho(A-BK_t))} \big) \leq O\big( \frac{\ln(n/\epsilon)}{1-\rho(A-BK_t)} \big).
    $
    Finally, combining~\cref{cor:qgp_a1},~\cref{thm:a1_OPN}, and~\cref{prop:a2_OPN}, we need ${\delta} \leq O(\epsilon)$,
    and so we choose ${\delta}$ so that $\delta \leq O(\epsilon)$ (this will not affect the complexity in terms of big-O notation, as shown below).

    Altogether, the total number of samples in each call to~\cref{alg:a2_OPN} is
    \begin{align*}
        N_S \cdot \tau 
        &\leq
        O\Big( \frac{\|K_t\|^8_F \cdot J(K_0)^{10}  \cdot m^4(n+m)[\ln(1/ \delta) + \ln(n/\epsilon)]^{6} }{(1-\rho(A-BK_t))^3\epsilon} \Big) \\
        &\leq
        O( {J(K_0)^{21} \cdot m^4(n+m) [\ln(1/\delta) + \ln(n/\varepsilon)]^{6} }/{\varepsilon}),
    \end{align*}
    where the second line used~\cref{cor:qgp_a1} and $(1-\sqrt{1-u})^{-1} \leq 2u^{-1}$ for $u \in (0,1)$, 
    $\epsilon = O(\varepsilon/(J(K_0)^3\|K_t\|_F^2)) \geq O(\varepsilon/J(K_0)^4)$,~\cref{lem:lqgp_a13}, and $J(K_t) \leq 2J(K_0)$ from~\cref{thm:b1_LQR} (and assumption $\varepsilon \leq J(K_0)$). Now, since~\cref{alg:a2_OPN} is called during each of the $T = O(J(K_0)^2 \ln(1/\varepsilon))$ iterations in~\cref{alg:b1_LQR}, then the total number of samples is $\bar{N} = N_S \cdot \tau \cdot T$, and the probability of success is $1 - T [2(N_S+S) \delta + 2(N_S+S) \delta] \geq 1-8\bar{N}{\delta}$,
    where the first term $2(N_S+S){\delta}$ is the failure from~\cref{prop:a2_OPN} and the second term is from event $\mathcal{E}(\delta)$ from~\eqref{eq:qgp_a57} not occurring.
    This completes the proof for estimating
    every natural gradient $E_{K_t}$ up to accuracy $\epsilon'$, and from~\cref{thm:b1_LQR} it guarantees $J(K_T) - J(K^*) \leq \varepsilon$.
\end{proof}
\hide{}{
\begin{proof}
    \edits{Our goal is to apply~\cref{thm:b1_LQR}. To do so, we need to verify the natural gradient estimation error is at most $\epsilon' := \min\{C_4, C_5, C_6\}$, where the constants, as defined in the aforementioned theorem, are on the order of
    \begin{align*}
        C_4 &= \big(\frac{\sigma_{\min}(Q) \sigma_{\min}(\Psi)C_1}{32\|B\|} \big)^2 \frac{1}{J(K_0)^2} = O\big(\frac{1}{J(K_0)^2}\big) \\
        C_5 &= \big(\frac{\sigma_{\min}(\Psi_\sigma)C_2}{3840\|B\|C_3^2}\big)^{2/3} \cdot \frac{\varepsilon^{2/3}}{J(K_0)^{2/3}} = O\big(\frac{\varepsilon^{2/3}}{J(K_0)^2}\big) \\
        C_6 &= \min\big\{ \frac{C_2^2\varepsilon}{120\sigma_{\min}(\Psi)} \cdot \frac{1}{J(K_0)J(K^*)}, \frac{C_2 \varepsilon }{240C_1C_3J(K_0)} \big\} 
        = O \big(  \frac{\varepsilon}{J(K_0)^3} \big).
    \end{align*}
    Hence, $\epsilon' = O(\varepsilon/J(K_0)^3)$. Hence, once we show $\|E_{K_t} - E_{K_t}^\star\|_F^2 \leq \epsilon'$, then we can guarantee the final iterate $K_T$ satisfies $J(K_T) - J(K^*) \leq \varepsilon$.}

    \edits{Now, we will establish the accuracy needed for evaluating the natural gradient in~\cref{alg:a2_OPN}. Recall from~\eqref{eq:natgrad},~\eqref{eq:phi_def}, and~\eqref{eq:linsys_bellman} the relation between the natural gradient $E_{K_t}$ and the vector $\vartheta(K_t)$ solve by~\cref{alg:a2_OPN}:
    \begin{align*}
        &\|\overbrace{\Theta(K_t)^\star_{22}K_t - \Theta(K_t)^\star_{21}}^{E_{K_t}^\star} - \overbrace{\Theta(K_t)_{22}K_t - \Theta(K_t)_{21}}^{E_{K_t}} \|_F^2 \\
        &\leq
        2\|\Theta(K_t)^\star_{22} - \Theta(K_t)_{22}\|_F^2\|K_t\|_F^2 + 2\|\Theta(K_t)^\star_{21} - \Theta(K_t)_{21}\|_F^2 \\
        &\leq
        2(\|K_t\|_F^2+1)\|\theta(K_t)^\star - \theta(K_t)\|_2^2 \\
        &\leq
        2(\|K_t\|_F^2+1)\|\vartheta(K_t)^\star - \vartheta(K_t)\|_2^2
    \end{align*}
    Hence, we need $\|\vartheta(K_t)^\star - \vartheta(K_t)\|_2^2 \leq \epsilon := \epsilon'/(2(\|K_t\|_F^2+1))$ from~\cref{alg:a2_OPN}.}
    
    \edits{Let us now specify the total number of iterations $N_S$, mixing time $\tau$, and failure rate $\bar{\delta}$ (c.f., $\delta$ is the chosen failure rate in the statement of the theorem) used in~\cref{prop:a2_OPN}\footnote{For easy of exposition, we only show polynomial dependence on the cost $J(K)$, control dimension $k$, norm of policy $\|K\|_F$, and accuracy $\epsilon$. We also show logarithmic dependence on $\epsilon$ and the state dimension $n$.} to achieve the desired accuracy $\epsilon$. From $M_X$ and $M_Y$ in~\cref{thm:a1_OPN} combined with~\cref{lem:a1_LQR} and the assumption $\|\vartheta_0\| = O(R_*) = O(J(K))$ (\cref{lem:lqgp_a30}), then 
    $M_X = M_Y = O\big((J(K)^2+k)^2\|K\|^4_F\big( \ln \frac{1}{\delta} \big)^\beta \big)$ 
    where $\beta = 5/4$ (\cref{lem:lqgp_a22}). Since we only shrink the feasible region, then $D_{X_s} \leq D_0 = O(R_*) = O(J(K))$. According to~\cref{lem:a6_OPN}, $\|H\| = O(\|K\|_F^2J(K) + k)$. Finally, we have $\mu^{-1} \equiv \sigma_{\min}(H)^{-1} = O([\|K\|_F^2J(K) + k]J(K))$ via~\cref{lem:lqgp_a22}. Putting everything together,
    \begin{align*}
        N_S &= O\Big( \frac{\|K\|^8 J(K)^2[J(K)+k]^6(\ln(1/\delta))^{7/2}}{\epsilon} \Big).
    \end{align*}
    To bound $\tau$, we observe from~\cref{lem:lqgp_a22} that $C = O(\sqrt{n})$. Then using~\cref{cor:qgp_a1} and the simple bound $\frac{1}{1-\sqrt{1-u}} \leq 2/u$ for any $u \in (0,1)$, 
    \begin{align*}
        \tau &= O\big( \frac{\ln(n/\epsilon)}{\log(1/\rho)} \big) = O\big( \frac{\ln(n/\epsilon)}{1-\rho} \big)  \leq O(J(K) \ln(n/\epsilon)) 
    \end{align*} 
    Combining $O_X$ and $O_Y$ from~\cref{thm:a1_OPN} with the bounds in~\cref{cor:qgp_a1}, we conclude $O_X = O_Y = O(n^2)$. Then we require $\bar{\delta} \leq O(\epsilon/n^4)$, and so we set $\bar{\delta} = \min\{\delta, O(\epsilon/n^4)\}$. }

    \edits{Putting everything together, the total number of samples in each call to~\cref{alg:a2_OPN} is
    \begin{align*}
        N_S \cdot \tau 
        &=
        O\Big( \frac{\|K_t\|^8 J(K_t)^3[J(K_t)+k]^6(\ln(1/\delta) + \ln(n/\epsilon))^{9/2} }{\epsilon} \Big) \\
        &\leq
        O\Big( \frac{\bar{K}^{10} J(K_0)^6[J(K_0)+k]^6(\ln(1/\delta) + \ln(n/\varepsilon))^{9/2} }{\varepsilon} \Big),
    \end{align*}
    where the last line uses $\epsilon = O(\varepsilon/(J(K_0)^3\|K_t\|_F^2))$, definition of $\bar{K}$, and $J(K_t) \leq 2J(K_0)$ from~\cref{thm:b1_LQR}. Now, since~\cref{alg:a2_OPN} is called during each of the $T$ iterations in~\cref{alg:b1_LQR}, then the total number of samples is $\bar{N} \cdot T = N_S \cdot \tau \cdot \ln(1/\varepsilon)$. Moreover, the probability of success is
    \begin{align*}
        1 - O(\ln(1/\varepsilon)) [\overbrace{2(N_S+S)\delta }^{\text{failure of~\cref{prop:a2_OPN}}} + \overbrace{2(N_S+S) \delta}^{\text{event $\mathcal{E}$~\eqref{eq:qgp_a57} does not occur}}] 
        &\geq
        1-O(\bar{N}\ln(1/\varepsilon))\delta, 
    \end{align*}
    which finishes the proof.}
\end{proof}
}

Some comments are in order. 
Assuming $\varepsilon \leq J(K_0)$ is mild, otherwise the initial $K_0$ suffices since $J(K_0)-J(K^*) < \varepsilon$. 
So the only non-trivial assumption we make is $K_0$ being stable. In contrast, prior works either arbitrarily postulate a uniform bound for $\bar{K} = \max_{0 \leq t \leq T-1} \|K_t\|$ and $\rho_T = \max_{0 \leq t \leq T-1} \rho(A-BK_t)$~\cite{zeng2021two,zhou2022single} or have their sampling complexity depend on $\bar{K}$ and $\rho_T$~\cite{yang2019provably}. We do neither. Instead, our algorithm ensures $\bar{K}^2$ and $\rho_T$ are bounded by $O(J(K_0))$. This is why our dependence on $J(K_0)$ is so large. 
We believe this dependence may be overconservative, as in our experiments we observed a less acute dependence on $J(K_0)$.  

\section{Numerical Experiments} \label{sec:exper}
We now run numerical experiments with our actor-critic algorithm. The code is at \texttt{https://github.com/jucaleb4/online-lqr}. First, let us describe the implementation details.

\subsection{Implementation details} \label{sec:implementation_details}
We implemented three algorithms: multi-epoch natural policy gradient (NPG), which is the combination of NPG~\cref{alg:b1_LQR} (actor) with the multi-epoch scheme~\cref{alg:a2_OPN} (critic); 
single-epoch NPG which uses the single-epoch critic~\cref{alg:a1_OPN};  
and two-time scale actor-critic (AC) method~\cite{zeng2021two}, which achieves the sampling complexity of $\tilde{O}(\varepsilon^{-3/2})$ for the online (single trajectory) setting.
We fine tuned all methods using a simple grid search similarly to~\cite{lan2012validation}, which includes the parameters (whenever applicable to the method) for the step sizes, mixing times $\tau$, initial diameter $D_0$, and number of epochs $S$. 
See our code for the exact parameter values.
We also applied a mini-batch to the gradient estimates for both the single- and multi-epoch scheme (but not the two-time scale actor-critic method) to reduce variance and produce higher accuracy results.
More specifically, our mini-batch scheme collects $m_{\text{batch}}=10$ (found by manually tuning) estimates of the gradient, each of which involves $\tau$ samples, sequentially along the same trajectory and then averages the $m_{\text{batch}}$ estimates to form a single estimate. 
Thus, the number of samples for a single gradient increases by a factor of $m_\text{batch}$.

\subsection{Simple synthetic problem}
Consider the synthetic problem from~\cite{dean2020sample}:
\begin{align*}
    A = \begin{bmatrix}[r] 1.01 & 0.01 & 0 \\ 0.01 & 1.01 & 0.01 \\ 0 & 0.01 & 1.01 \end{bmatrix},
    ~~
    B = R = \Psi = I_3,
    ~~
    Q=10^{-3}I_3,
\end{align*}
where $I_3$ is the identity matrix of dimension $3 \times 3$. We also set the added noise as $\sigma=1$. 
We call this problem the \textit{simple} environment.
The initial controller is set as $I_3$.
We also implemented a \textit{large simple} environment, which enlarges the Toeplitz matrix and identity matrices from the simple environment from dimension $n=3$ to $n=100$, to test a larger problem.
All constants remain the same.

The results from both simple environments are shown in the two left hand side plots of~\cref{fig:simple_env}. 
While the two-time scale actor-critic (AC) method has fast initial convergence, a majority of the seeds (19/32) yield unstable policies after 1200 samples of the simple environment.
In contrast, both the single- and multi-epoch NPG decrease the objective further and more gradually, and moreover, all the seeds output a stable policy.
The robust performance of both NPG methods highlights the importance of removing the assumption that the policy is almost surely stable at every iteration (see the discussion after~\cref{thm:a3_OPN}).
Moreover, the multi-epoch NPG has equal or better performance than the single-epoch.
This further supports the theoretical advantage of the former method over the latter (\cref{prop:a2_OPN}).
Finally, both NPG methods exhibit fast initial convergence followed by slower and stagnated convergence for the simple and large simple environments, respectively.
This slower performance is from the policy evaluation (i.e., critic) error, which propagates to the actor and stalls convergence.

\begin{figure}[htbp]
  \centering
  \includegraphics[width=\textwidth]{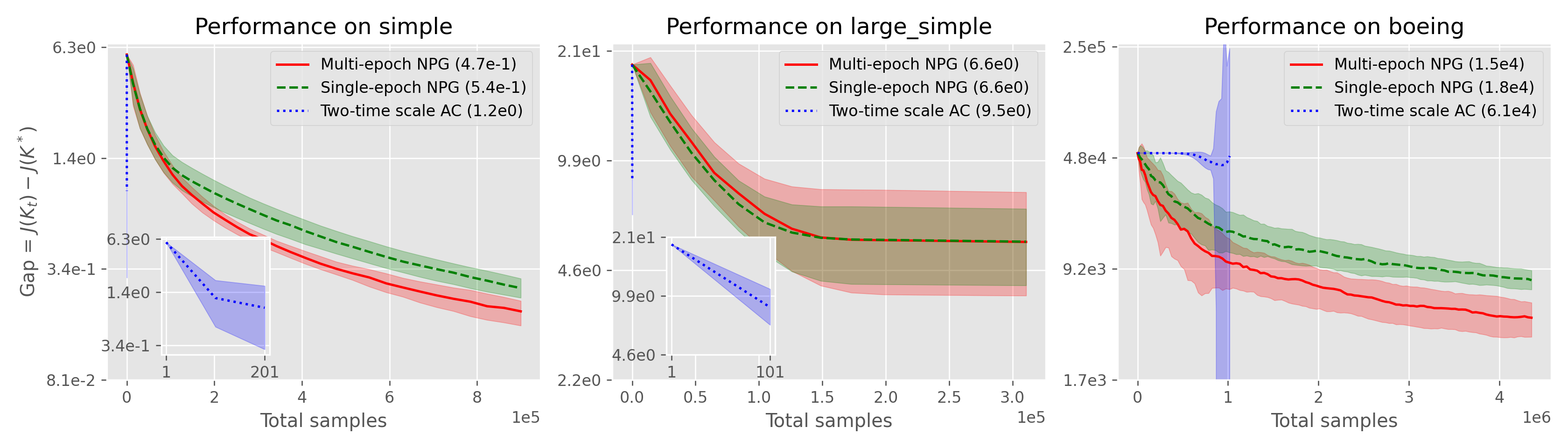}
\caption{The line and shaded region are the median and confidence interval, respectively, of the function gap $J(K_t)-J(K^*)$ (on a log scale) w.r.t.~the total samples over 32 seeds. 
The function gap is shown when a majority (i.e., $\geq60$\%) of the seeds have a stable policy.
Since the two-time scale AC does not have a stable policy in a majority of seeds after a few hundred samples in the simple environments, \edits{we add zoomed-in plots to better display their performance}.
In contrast, both of our NPG methods output stable policies in every seed.
The median last-iterate performance $J(K_t)$ is in the legend. The optimal $J(K^*)$ (left to right) is approximately 0.28, 0.93, 11200.}
\vspace{-1em}
\label{fig:simple_env}
\end{figure}

\subsection{Longitudinal control of a wide-body aircraft}
We consider a state-feedback control problem of the longitudinal dynamics of a Boeing 747 aircraft. The continuous time nonlinear dynamics are linearized around trim conditions and subsequently discretized as in~\cite{kotsalis2021convex}. The resulting matrices in the state recursion are 
\begin{align*}
    A 
    &= 
    \begin{bmatrix}[r]
    1 & -1.13 & -0.65 & -8.07 & 1.59 \\
    0 & 0.77 & 0.32 & -0.98 & -2.97 \\
    0 & 0.12 & 0.02 & -0.00 & -0.36 \\
    0 & 0.01 & 0.01 & -0.03 & -0.04 \\
    0 & 0.14 & -0.09 & 0.29 & 0.76
    \end{bmatrix},
    B 
    = 
    \begin{bmatrix}[r]
    89.20 & -50.17 & 1.13 &-19.35 \\
    5.22 &  6.36 & 0.23 & -0.32 \\
    -9.47 &  5.93 &-0.12 &  0.99 \\
    -0.32 &  0.32 & -0.01 & -0.01 \\
    -4.53 &  3.21 & -0.14 &  0.09
    \end{bmatrix}.
\end{align*}
Here, the $n$-dimensional ($n=5$) state is $x_t = [\alpha_t, v_t, v^\perp_t, \theta_t, \phi_t]$, where the variables correspond, respectively, to the deviation in altitude (positive is down), deviation in velocity along the aircraft's axis (forward is positive), deviation in velocity orthogonal to the aircraft's axis (positive is down), deviation in angle between the aircraft's axis and the X-axis, and angular velocity of the aircraft (pitch rate). 
We set the cost matrices $Q$ and $R$ as identity matrices, controller noise as $\sigma=1$, and covariance matrix as a symmetric positive definite matrix that attempts to capture dependencies in the noise (e.g., a positive correlation between the deviation in altitude $\alpha_t$ and orthogonal velocity $v^\perp_t$ from wind pushing the plane down).
Due to space constraints, we omit the matrix (it can be found in the code).
Finally, a matrix of all ones scaled by 1/200 is used as the initial stabilizing controller.

The results are shown in the right hand side of~\cref{fig:simple_env}. Similar to the simple synthetic problems, the single- and multi-epoch NPG can decrease the function gap at a near linear rate before slowing down. However, this problem requires about five times more samples than the simple synthetic problems to achieve a similar relative decrease in the function gap. This can be possibly explained by the large initial value $J(K_0) = 52499$, which affects the sampling complexity (\cref{thm:a3_OPN}). 
Overall, the multi-epoch NPG outperforms the single-epoch NPG, while the two-time scale actor-critic can only make a small improvement before diverging.

\section{Conclusion}
When applied to LQR, our proposed actor-critic method outputs a policy $\hat{K}$ such that $J(\hat{K}) - J(K^*) \leq \varepsilon$ with high probability and $\tilde{O}(\varepsilon^{-1})$ sampling complexity. 
The actor uses natural policy gradient~\cite{kakade_NIPS_01}, and we provide a novel analysis on the effect of the gradient estimation error. 
The critic is our novel conditional stochastic primal-dual method. 
Unlike previous model-free methods, our method only samples the system along a single trajectory (i.e., the online setting), and it does not reset the system nor assume some arbitrary behavior from the algorithm. In other words, we match the best rate from prior works (up to polylogarithmic factors) for a harder problem and with less restrictive assumptions.

While our work studies the average reward setting, our work should be amenable to the infinite-horizon discounted setting as well, since they both share some properties such as the P\L{}-condition and the natural gradient being partially defined by a solution to a linear system of equations. 


\section*{Acknowledgments}
We thank Sihan Zeng for pointing us to~\cite{zeng2021two,zhou2022single} and discussing the results therein. 


\bibliographystyle{siamplain}
\bibliography{references}

\appendix
\section{Proofs from~\cref{sec:a4_LQR}}  \label{sec:pfs_for_basic_props}
\begin{proof}[Proof of~\cref{lem:lqgp_a13}]
    The first three inequalities follow from the definitions of $\Sigma_K$, $P_K$, and $J(K)$, respectively, and basic trace inequalities (c.f.~\cite[Lemma 3.8]{fatkhullin2021optimizing}). 
    
    To show the last inequality, we first define the identity operator $I$ and linear symmetric operator $\mathcal{T}(X) := (A-BK)X(A-BK)^T$. 
    Since $\rho(A-BK) < 1$, then the Neumann series $(I-\mathcal{T})^{-1}(X) = \sum_{t \geq 0}(A-BK)^tX[(A-BK)^T]^t$ is well-defined~\cite[Lemma 18]{fazel2018global}.
    Denote $\mathcal{T}^t$ as the composition of $\mathcal{T}$ with itself $t$ times.
    Now, by definition of $\Sigma_K$ in~\eqref{eq:lyap_cov_mat}, we have $(I - \mathcal{T})(\Sigma_K) = \Psi + \sigma^2 BB^T$, and so
    \begin{align*}
        \sigma_{\min}(Q)^{-1} J(K)
        \geq
        \trace[\Sigma_K]
        &= 
        \trace[(I-\mathcal{T})^{-1}(\Psi + \sigma^2 BB^T)] \\
        &\geq
        \sigma_{\min}(\Psi) \trace[(I-\mathcal{T})^{-1}(I)] \\
        &=
        \textstyle \sigma_{\min}(\Psi) \sum_{t \geq 0} \trace[\mathcal{T}^t(I)] \\
        &\geq
        \textstyle \sigma_{\min}(\Psi) \sum_{t \geq 0} \rho(A-BK)^{2t} \\
        &=
        \sigma_{\min}(\Psi)/(1-\rho(A-BK)^2),
    \end{align*}
    where the third line is by definition of $(I-\mathcal T)^{-1}$ and exchanging summations, while the fourth line is from the inequality $\trace[(A-BK)^t((A-BK)^T)^t] \geq \|(A-BK)^t\|^2 \geq \rho(A-BK)^{2t}$.
\end{proof}
\section{Proofs from~\cref{sec:a1_LQR}} \label{sec:b7_LQR}

\begin{proof}[Proof of~\cref{lem:a1_LQR}]
    Recall from~\eqref{eq:phi_def} that $\phi_t \equiv \phi(x_t,u_t) = \svec( z_tz_t^T )$ and $z_t = [x_t,u_t]$. 
    Using the definition of the stochastic estimate $\tilde{H}_t$ for $H$ from~\eqref{linear_system_parameters},
        $\|\tilde{H}_t\| 
        \leq 
        2 + 2\|z_t\|^4 + \|z_t\|^2\|z_{t+1}\|^2  
        \leq
        2 + 4 \max\{\|z_t\|^4, \|z_{t+1}\|^4\}$. 
    By definition of $\mathcal{E}_t(\delta)$ in~\eqref{eq:a1_LQR} and the covariance matrix $\TSig_K^{(t)}$ from~\eqref{eq:a15_LQR}, then for any $t$,
    \begin{equation} \label{eq:a22_LQR}
    \begin{split}
        \|z_t\|_2^2 
        &\leq
        4\Big( \frac{c_2^2}{\sqrt{c_1}}\|\TSig_K^{(t)}\| + \frac{c_2^2}{c_1} \trace(\TSig_K^{(t)}) + \|z_0\|^2 \Big) \ln\frac{1}{\delta} \\
        &\leq
        4(1+\|K\|)^2\Big( \frac{c_2^2}{\sqrt{c_1}}\|\Sigma_K^{(t)}\| + \sigma^2 + \frac{c_2^2}{c_1} \trace(\Sigma_K^{(t)}) + \sigma^2 \cdot m + \|x_0\|^2 \Big) \ln\frac{1}{\delta} \\
        &\leq 
        4(1+\|K\|)^2\Big( \frac{c_2^2(\sqrt{c_1}+1)}{c_1} \cdot \frac{J(K)}{\sigma_{\min}(Q)} + \sigma^2 \cdot (m+1) + \|x_0\|^2 \Big) \ln\frac{1}{\delta},
    \end{split}
    \end{equation}
    where the last line used $\|\Sigma_K^{(t)}\| \leq \trace(\Sigma_K^{(t)}) \leq \trace(\Sigma_K) \leq J(K)/\sigma_{\min}(Q)$ thanks in part to~\cref{lem:lqgp_a13}. Combining the last two inequalities gets us the upper bound on $\|\tilde{H}_t\|$. One can similarly bound $\|\tilde{b}_t\|$ by recalling $c(x_t,u_t) = x_t^TQx_t + u_t^TRu_t$. 
%
\end{proof}

In order to prove~\cref{lem:lqgp_a22}, 
we will first need the geometrically fast mixing property of a linear dynamical system (e.g., evolution of the state-action pair $[x_t,u_t]$ induced by~\eqref{eq:randomized_linear_state_feedback} and~\eqref{eq:a17_LQR}).
\begin{proposition}[Proposition 3.1~\cite{tu2018least}] \label{prop:a1_LQR}
    Take the linear dynamical system $X_{t+1} = FX_t + w_t$, where $\omega_t \sim \mathcal{N}({0}, \Lambda)$ and $\Lambda \succ 0$. Suppose $\left\|F^k\right\| \leq \Gamma \rho^k$ for all $k \geq 0$, where $\Gamma>0$ and $\rho \in(0,1)$. Let $\mathbb{P}_{X_k}\left(\cdot \mid X_0=x\right)$ be the conditional distribution of $X_k$ given $X_0=x$. Then for all $k \geq 0$ and any distribution $\nu_0$ over $x \in \mathbb{R}^n$,
    \begin{align*}
        \mathbb{E}_{x \sim \nu_0}\left[\left\|\mathbb{P}_{X_k}\left(\cdot \mid X_0=x\right)-\nu_{\infty}\right\|_{\mathrm{tv}}\right] 
        \leq 
        \frac{\Gamma}{2} \sqrt{\mathbb{E}_{\nu_0}\left[\|x\|^2\right]+\frac{n}{1-\rho^2}} \rho^k,
    \end{align*}
    where $\|\cdot\|_{\mathrm{tv}}$ is the total-variation distance.
\end{proposition}
We need to bound higher moments of a mean-zero Gaussian random vector. 
\begin{lemma} \label{lem:lqgp_a50} 
    Let $X \sim \mathcal{N}(0, \Sigma)$ and $I = \{i_j \in [n]\}_{j=1}^{2p}$ be a set of arbitrary indices for some positive integer $p$. Then
    $
        \mathbb{E}[\prod_{j=1}^{2p} X_{i_j}] 
        \leq 
        \big( 3 \cdot 5 \cdot \ldots \cdot (2p-1) \big) \|\Sigma\|^p$.
\end{lemma}
\begin{proof}
    Since $\vert I \vert = 2p$ is even, Isserlis' Theorem~\cite{isserlis1918formula,withers1985moments} derives us the closed-form expression $\mathbb{E}[\prod_{j=1}^{2p} X_{i_j}] = \sum\limits_{\sigma \in \Pi(I)} \prod_{(i,j) \in \sigma} \mathrm{Cov}(X_i X_j) \leq \vert\Pi(I)\vert \cdot \|\Sigma\|^p$,
     where $\Pi(I)$ denotes the set of all partitions of $I$ into (disjoint) pairs. Here, it is known the cardinality of $\Pi(I)$ is $\vert \Pi(I) \vert = (2p)!/(2^p p!) = 1 \cdot 3 \cdot \ldots \cdot (2p-1)$.
\end{proof}
\begin{proof}[Proof of~\cref{lem:lqgp_a22}]
    Let $z_t = [x_t, u_t]$ be the state-action pair and $\mathrm{d}P_{t' \vert t}$ be the probability density function for $z_{t'}$ conditioned on $z_t$. Recall $H = \mathbb{E}_{\xi \sim \Pi_K}[\tilde{H}(\xi)]$, and denote $\tilde{H}_{1}$ as the stochastic estimate of $\tilde{H}$ without the constant value ``1'' in the leftmost column of the first row. Writing $\mathcal{E}_{t+\tau} \equiv \mathcal{E}_{t+\tau}(\delta)$, 
    then
    \begin{align*}
        &\|H - \mathbb{E}[\tilde{H}_{t+\tau} \vert \mathcal{F}_{t-1}, \mathcal{E}_{t+\tau}] \|_* \\
        &=
        \Big \| \int_{\xi \in \mathcal{E}_{t+\tau}} \tilde{H}_1(\xi) [\mathrm{d}\Pi_K(\xi) - \mathrm{d}P_{t + \tau \vert t }(\xi)] 
        + \int_{\xi \not \in \mathcal{E}_{t+\tau}} \tilde{H}_{1}(\xi) \mathrm{d}\Pi_K(\xi)
        \Big\|_* \\
        &\leq
        \int_{\xi \in \mathcal{E}_{t+\tau}} \| \tilde{H}(\xi) \|_* \vert \mathrm{d}\Pi_K(\xi) - \mathrm{d}P_{t+\tau \vert t}(\xi) \vert  
        +
        \int_{\xi \not \in \mathcal{E}_{t+\tau}} \| \tilde{H}_{1}(\xi) \|_* \mathrm{d}\Pi_K(\xi)
        \\
        &\stackrel{(i)}{\leq}
        \Big( \frac{M_H}{2}\sqrt{\|z_t\|^2+\frac{n+m}{1-\rho^2}} \Big) \ln\frac{1}{\delta} \Big) \rho^\tau 
        +
        \sqrt{\mathbb{E}_{\Pi_K} \|\tilde{H}_1\|^2 } \sqrt{1-\mathrm{Pr}\{\mathcal{E}_{t+\tau}\}} \\
        &\stackrel{(ii)}{\leq}
       \frac{M_H}{2}  \Big( M_H^{1/4} + \sqrt{\frac{n+m}{1-\rho}} \Big) \Big( \ln\frac{1}{\delta} \Big)^{3/2} \rho^\tau
       +
       \sqrt{\mathbb{E}_{\Pi_K} \|\tilde{H}_1\|^2 } \sqrt{\delta},
    \end{align*}
    where the bound on the first term in (i) is by~\cref{prop:a1_LQR} (where $\rho \equiv \rho(A-BK) < 1$ by assumption $K \in \mathcal S$) and~\cref{lem:a1_LQR} and the second term by Cauchy-Schwarz. Inequality (ii) is by $\|z_t\|^2 \leq \sqrt{M_H} \ln (1/\delta)$, where we used~\eqref{eq:a22_LQR} and $M_H$ from~\cref{lem:a1_LQR}, as well as~\cref{prop:pqgp_a9} and~\eqref{eq:a1_LQR} to bound the probability term. 
    By definition of $\tilde{H}(\xi)$ in~\eqref{linear_system_parameters} and denoting $\mathbb{E} = \mathbb{E}_{\Pi_K}$,
    \begin{align*}
        &\mathbb{E}\|\tilde{H}_1\|^2  
        \leq
        \mathbb{E}\|\tilde{H}_1\|^2_F \\
        &\leq
        10 (n+m)^2 \max_{i,j,u,v} \Big\{ \mathbb{E} ([z_t]_i [z_t]_j)^2, \mathbb{E} ([z_t]_i [z_t]_j [z_t]_u [z_t]_v)^2, \mathbb{E} ([z_t]_i [z_t]_j [z_{t+1}]_i [z_{t+1}]_v)^2 \Big\} \\
        &\leq
        40 (n+m)^2 \max_i\{ \mathbb{E}([z_t]_i)^4, \mathbb{E}([z_t]_i)^8, \mathbb{E}([z_{t+1}]_i)^8 \} \\
        &\leq
        \big( 65 (n+m) \max\{\|\TSig_K\|^2, \|\TSig_K\|^4\}  \big)^2 \equiv O_H^2,
    \end{align*}
    where the last line is by~\cref{lem:lqgp_a50}. Combining the last two results yields the bound on the bias for $\tilde{H}$. 
    By noting the similarity in definition between $\tilde{H}$ and $\tilde{b}$ in~\eqref{linear_system_parameters}, one can similarly bound $\|b - \mathbb{E}[b_{t+\tau} \vert \mathcal{F}_{t-1}]\|_*$. 
\end{proof}

\begin{proof}[Proof of~\cref{lem:lqgp_a29}]
Recall the matrix $H = \mathbb{E}_{\Pi_K}[\tilde{H}]$ from~\eqref{linear_system_parameters}. 
Our goal is to show the smallest singular value of $H$, denoted by $\sigma_{\min}(H)$, is greater than zero, or upper bound $\|H^{-1}\| = (\sigma_{\min}(H))^{-1}$.
Assuming this is true, because $\|H\|$ is bounded (by $K \in \mathcal S$ and~\cref{lem:a6_OPN}), then $H$ is invertible.
When $H$ is invertible, then there exists a $\vartheta^*$ such that $H\vartheta^*=b$, which
in view of the discussion prior to the statement of~\cref{lem:lqgp_a29}, 
implies $g([\vartheta,y]) \geq f(\vartheta) = \|H\vartheta-b\| \geq \sigma_{\min}(H)\|\vartheta-\vartheta^*\|$.
Therefore, when $\sigma_{\min}(H)$ is positive, the sharpness constant $\mu$ can be set to $ \sigma_{\min}(H)$.

Towards that, using the condition $\sigma^2 \geq \sigma_{\min}[1 + \|K\|^2]$ and the policy $K \in \mathcal{S}$ within the proof of~\cite[Lemma B.2]{yang2019provably}, one can show
\begin{align*}
    \|H^{-1}\|^2 
    &\leq 
    \frac{1}{(1-\rho^2)^{2}} \Big( \frac{1}{\sigma_{\min}(\Psi)^{4}} + \frac{n+m}{\sigma_{\min}(\Psi)^{2}} + (1-\rho^2) \Big) 
    <
    \frac{(n+m)(1+\frac{1}{\sigma_{\min}(\Psi)^2})^2}{(1-\rho^2)^2},
\end{align*}
where we used the fact $\rho \equiv \rho(A-BK) \in (0,1)$ since $K \in \mathcal{S}$. 
\end{proof}

\end{document}